\pdfoutput=1
\documentclass[12pt]{article}

\usepackage{graphicx} 
\usepackage{pgfplots}
\usepackage[cmex10]{amsmath}
\usepackage{amsfonts}
\usepackage{mathrsfs}
\usepackage{amsthm}        
\usetikzlibrary{arrows}
\usepackage{tikz-cd}
\usepackage{quiver}
\usepackage[a4paper,left=30mm,right=30mm,top=25mm,bottom=25mm]{geometry}

\theoremstyle{plain} 
\newtheorem{theorem}{Theorem}[section]

\newtheorem{lemma}[theorem]{Lemma}

\newtheorem{problem}[theorem]{Open problem}

\theoremstyle{definition} 
\newtheorem{remark}[theorem]{Remark}
\newtheorem{definition}[theorem]{Definition}
\newtheorem{claim}{Claim}[theorem]

\newtheorem{example}[theorem]{Example}

\newcommand{\comment}[1]{\typeout{swallowing comment}}

\title{Algebraic characterisation of pseudo-elementary and second-order classes}
\author{János Balázs Ivanyos}

\begin{document}

\begin{center}

{\LARGE {Algebraic characterisation of pseudo-elementary and second-order classes}} \\[0.7cm]
{\large János Balázs Ivanyos}\\[0.2 cm]

{ 2026 May}\\[1 cm]

\end{center}
\begin{abstract}
In this paper, we provide  purely model-theoretic (algebraic) characterisations for classes definable in second-order logic and for pseudo-elementary classes 
(including PC and PC$_{\Delta}$ classes). Classical results of this flavor include Keisler–Shelah type theorems (characterising first-order definability by closure under ultraproducts and ultraroots) and Birkhoff's \textbf{HSP} theorem; a key starting point for this paper is Sági's work \cite{cl1}, which provides an algebraic description of classes definable by existential second-order sentences. Here we resolve several open problems from \cite{limit prod}
and \cite{cl1}. \\
\indent Our main results are the following.
\begin{itemize}
   \item We solve the long-standing problem of giving a purely algebraic characterisation of pseudo-elementary classes: we
characterise PC$_{\Delta}$ classes by intrinsic closure properties.  We also give a  characterisation for the basic pseudo-elementary classes
(PC). 
   \item We provide a structural classification of second-order equivalent structures, and we obtain purely algebraic characterisations
of the classes definable by second-order formulas as well as those definable by finitely many second-order sentences.
\end{itemize}

\end{abstract}
\newpage

\tableofcontents
\newpage

\section{Introduction}

One of the fundamental tools of classical (first-order) model theory is the ultraproduct construction. As it is well known, with the aid of this construction, one can characterize (first-order) axiomatizable classes of structures. Further, a key result that reveals the strength and elegance of the ultraproduct construction is the celebrated ``Isomorphic Ultrapowers Theorem" of Keisler and Shelah. It states that two structures ${\cal A}$ and ${\cal B}$ are elementarily equivalent (that is, ${\cal A}$ and ${\cal B}$ satisfy the same first-order formulas) if and only if ${\cal A}$ and ${\cal B}$ have isomorphic ultrapowers. Assuming the generalized continuum hypothesis, Keisler proved this theorem in 1968; later, in \cite{Shelah} Shelah eliminated the generalized continuum hypothesis from the proof. \\
\\
\indent 
According to the above theorems, the interconnections between first-order formulas and ultraproducts are well understood. However. The situation becomes substantially more complicated when one passes to higher-order logics. As {\bf a motivating example} we note that ultraproducts, in general, do not preserve the validity of second-order formulas. In fact, (in a fixed signature) one can always find finite structures with an infinite ultraproduct, therefore if a first-order formula is true in all finite structures, then it should also be true in some infinite structure, as well. Thus, in another words ``being finite'' and ``being infinite" are not first-order expressible. However, as it is also well known, these properties (finiteness and infinity) can be expressed by second-order formulas (see e.g. the proof of Theorem \ref{thm: NP} below). \\
\indent In view of the previous paragraph, it is natural to ask what can be said about higher order formulas and ultraproducts of models of them. This question was the starting point of 
S\'agi's paper \cite{cl1} and here we continue related investigations.

A little bit more precisely, here we continue investigations initiated
in \cite{cl1} and \cite{ultratop} and by developing
new technical tools, we will study connections between ultraproducts,
pseudo elementary classes and second-order formulas. In more detail, our
main
results are as follows.

\begin{itemize}

   \item We solve the long-standing  problem of giving a
purely algebraic characterisation of pseudo-elementary classes. In more detail, in Theorem \ref{thm: PC_DELTA} we
characterise PC$_{\Delta}$ classes by intrinsic closure properties.  We
also give a  characterisation for the basic pseudo-elementary classes
(PC) in Lemma \ref{lem: Pc-zárt}, and, using this description, in Theorem \ref{thm: NP} we obtain an internal (closure-based)
criterion for when a class of finite structures lies in \textsf{NP}.
   \item In Theorem \ref{thm: eqv} We provide a structural classification of second-order
equivalent structures, and in Theorem \ref{thm: axi so} we obtain purely algebraic characterisations
of the classes definable (axiomatizable) by second-order formulas. Further, in Theorem \ref{thm: axi finit so} we characterize classes
definable (axiomatizable) by finitely many second-order sentences.
\end{itemize}

From a technical perspective, our approach combines classical ultraproduct methods with tools from descriptive set theory. We will establish interconnections between axiomatizability and closed subspaces of certain spaces of formulas. This approach allows us to formulate preservation conditions in a topologically robust manner. More concretely, we will study certain subspaces of Cantor spaces associated to classes of structures with different axiomatizability properties, for the details we refer to Subsection \ref{spf}. We call these spaces ``spaces of formulas".

The structure of the rest of this work is as follows. In Section \ref{sec: prelim} we recall syntactical and semantical aspects of second-order logic, we provide a special Henkin-type sematics for second-order logic in which the second-order variables range over decomposable relations and functions, and finally we introduce the spaces of formulas we will deal with in later sections. In Section \ref{sec: pseudo} we provide a purely algebraic characterization of pseudo-elementary classes and in
\ref{sec: second-order} we characterize second-order elementary equivalence and second-order axiomatizable classes.
\section{Preliminaries}\label{sec: prelim}
In this section, we sum up some already known facts that we need later in this
work.

\subsection{Syntactical preparations}

We work with second-order logic over a given signature, allowing relation variables only. This entails no loss of generality: function variables of arity at most $n$ may be replaced by relation variables of arity at most $n+1$.

\begin{definition}
Fix a signature $\tau$. A second-order formula is in \emph{prenex second-order form} if it is of the
form
\[
Q_1\bar R_1\,Q_2\bar R_2\cdots Q_n\bar R_n\,\varphi,
\]
where each $Q_i\in\{\exists,\forall\}$ binds a finite tuple of relation
variables and $\varphi$ is first-order in the expanded signature. A block is a
maximal consecutive string of quantifiers of the same kind.

Let $\Sigma^1_0=\Pi^1_0$ be the class of first-order formulas, possibly with
free relation variables. For $n\geq1$, $\Sigma^1_n$ consists of formulas
equivalent to prenex formulas with $n$ alternating second-order blocks beginning
with an existential block; $\Pi^1_n$ is defined dually. Let $\Delta_n$ denote
the Boolean closure of $\Sigma^1_n\cup\Pi^1_n$. Thus
\[
SO=\bigcup_{n<\omega}\Sigma^1_n
 =\bigcup_{n<\omega}\Pi^1_n
 =\bigcup_{n<\omega}\Delta_n .
\]
We write $SO(\tau),\Sigma^1_n(\tau),\Pi^1_n(\tau),\Delta_n(\tau)$ and
$FO(\tau)$ when the signature may not be clear from the context.
\end{definition}

As it is well known, the second-order hierarchy and computational complexity theory are strongly connected.
In his 1974 paper \cite{Fagin},  Fagin proved that the complexity class \textsf{NP} (the class of decision problems decidable in nondeterministic polynomial time) is exactly the set of problems definable in existential second-order logic.
More precisely, a class $K$ of \emph{finite} $\tau$-structures belongs to \textsf{NP} if and only if there exists an existential second-order formula $\phi$ such that $$K=\mathrm{Mod}_{<\omega}(\phi).$$

In Section~\ref{sec: pseudo} we will need the notion of second-order types.
We shall also use the following notion of second-order type.

\begin{definition}\label{def: type}
Let $\kappa$ be uncountable, and let
$s=(s_i)_{i<\kappa}$ be a sequence of natural numbers in which each value
occurs $\kappa$ times. Let
$$\overrightarrow{X_{\kappa}}=\{X_i : i\in\kappa\}$$
be relation variables with $\mathrm{arity}(X_i)=s_i$. Let $F\subseteq\Delta_0$
be the set of closed formulas using only variables from $\bar X_\kappa$.
A set $p\subseteq F$ is a \emph{$\kappa$-2-type} if
\begin{itemize}
\item[(i)] every finite $q\subseteq p$ is satisfiable, i.e.
      $\exists\bar X_\kappa\,\bigwedge q$ has a model;
\item[(ii)] for every $\phi\in F$, either $\phi\in p$ or $\neg\phi\in p$.
\end{itemize}
\end{definition}

Finally, for regular $\kappa$ and cardinals $\lambda,\mu\leq\kappa$, the
language $L^2_{\kappa,\lambda,\mu}$ is obtained from atomic formulas by
allowing conjunctions and disjunctions of length $<\kappa$, first-order
quantification over tuples of length $<\lambda$, and second-order quantification
over tuples of relation variables of length $<\mu$. We shall mainly use
fragments of $L^2_{\kappa,\omega,\mu}$ with homogeneous second-order blocks.

\subsection{Henkin-type models}

For a structure $\mathcal A$ and a relation $R\subseteq A^k$, write
$\langle{\cal A},R\rangle$ for the expansion of {\cal A} obtained by adjoining $R$ to ${\cal A}$.

In this subsection we shall consider certain special models. These will be constructed using ultraproducts. A central tool for us throughout the paper will be the construction of the product of ultrafilters, we briefly recall what this construction is.

\begin{definition}
    Let \(\mathcal{F}\) be an ultrafilter on \(I\) and \(\mathcal{G}\) an ultrafilter on \(J\). The product ultrafilter \(\mathcal{F}\times\mathcal{G}\) on \(I\times J\) is defined by
\[
X\in\mathcal{F}\times\mathcal{G}\iff\{\,j\in J:\{\,i\in I:(i,j)\in X\,\}\in\mathcal{F}\,\}\in\mathcal{G}.
\]
\end{definition}

The following Fubini-type theorem from  Proposition 6.2.1 in \cite{chk}.

\begin{theorem}\label{thm: fubini}\footnote{This statement is formulated more generally for filters, and an even more general treatment can be found in \cite[Section 6.5]{chk}.
}
    Let us \(\mathcal{F}\) an ultrafilter on \(I\), \(\mathcal{G}\) another ultrafilter on \(J\) and a set of first-order structures $\{{\cal A}_{i,j}\text{ : }i\in I,\text{ }j\in J \}.$ Then $$\prod_{ I\times J}{\cal A}_{i,j}/{\cal F}\times{\cal G}\cong\prod_{J}\big(\prod_{I}{\cal A}_{i,j}/{\cal F}\big)/\cal G.$$
\end{theorem}

In Section~\ref{sec: second-order}, we will characterize second-order classes using ultrachains, so
recall the definition of ultrachain.
\begin{definition}
A sequence $\langle\mathcal A_\xi:\xi<\kappa\rangle$ is a
\emph{$\kappa$-long ultrachain} if, for each $\mu<\kappa$,
\[
\mathcal A_{\nu}\subseteq\mathcal A_{\nu+1}
\quad\text{and}\quad
\mathcal A_{\nu+1}\cong \mathcal A_\nu^{I_\nu}/\mathcal F_\nu
\]
whenever $\mu=\nu+1$, while
$\mathcal A_\mu=\bigcup_{\nu<\mu}\mathcal A_\nu$ for limit $\mu$.
Its limit is $\bigcup_{\xi<\kappa}\mathcal A_\xi$.
\end{definition}

\begin{definition}
    Fix a signature $\tau$, an ultrafilter $\cal F$ over the set $I$,  a set of $\tau$-structures $\{{\cal A}_{i} : i\in I\}$ and consider the ultraproduct ${\cal A}=\prod_{i\in I}{\cal A}_{i}/{\cal F}$. For $0<k$ a relation \(R \subseteq {}^{k}A\)  is defined to be
decomposable in the ultraproduct ${\cal A}=\prod_{i\in I}{\cal A}_{i}/{\cal F}$, iff for every \(i \in I\) there exists a  relation \(R_i \subseteq {}^{k}A_i\) such that $R=\prod_{i\in I}R_{i}/{\cal F}$. In this case, we also say, that \(R\) can be decomposed to \(R_i,\ i \in I\) or that \(R_i,\ i \in I\) is a decomposition of \(R\).
\end{definition}

\begin{remark}
    Note that the decomposability of a relation in a given structure \(\mathcal{A}\) depends on the ultrafilter and on the factors used to form the ultraproduct. For example, if $\cal F$ is a principal ultrafilter, then every relation on $\cal A$ is decomposable, however, as we will see, this is not true in general. Moreover, in Example~\ref{ex: np-példa} we will present a case where the same structure is obtained in two different ways (from different families of structures) using the same ultrafilter $\cal F$, yet the sets of decomposable relations produced by the two constructions are not the same.

\end{remark}

\begin{remark}\label{rem:ultradoboz}
   One may think of decomposable relations as ``ultra-boxes'', and in the limit case of decomposability one should think of boxes whose sides are constant on a large set.  Moreover, we remark that Sági \cite{cl1} introduces an interesting topological notion for decomposable relations, and the properties of these notions are discussed in detail in \cite{ultratop}.

\end{remark}

In the following, we will review certain \emph{Henkin-type models}. The notion of a ``general model'' was introduced by Henkin, motivated by  to obtain a completeness theorem for second-order logic. Models of this kind have since remained of central importance in the study of second-order logic \cite{vaa}. The basic idea is that quantifiers do not range over all relations, but only over a specified collection of them. Sági employed a similar construction in his paper \cite{cl1}, combining it with a topological aspect, as we remarked. For our purposes, the following so-called \emph{decomposable-Henkin models} will be of crucial importance.

\begin{definition}\label{Henkin model} Fix a signature $\tau$, a set of $\tau$-structures $\{{\cal A}_{i} : i\in I\}$  and let ${\cal A}=\prod_{i\in I}{\cal A}_{i}/{\cal F}$  an ultraproduct, then the  pair $\big({\cal A},\Upsilon\big)$ is a decomposable-Henkin model, where  $\Upsilon$ is the set of all decomposable relation on it.


Often, when it is not clear which ultraproduct a decomposable–Henkin model is associated with, we say ``the decomposable–Henkin model \(\big({\cal A},\Upsilon\big)\) formed from the family $\{ {\cal A}_{i}\text{ : } i\in I\}$ by the ultrafilter $\cal F$''.

Furthermore, if $K$ is a class of $\tau$-structures, then $\overline {K}$ denotes the class of all decomposable–Henkin \(\big({\cal A},\Upsilon\big)\) formed from some family $\{ {\cal A}_{i}\text{ : } i\in I\}$ by an ultrafilter $\cal F$, with $\{i\in I \text{ : } {\cal A}_{i}\in K\}\in{\cal F}$.

\end{definition}

\begin{remark}
    It is clear that if $K$ is a class of $\tau$-structures and $\big({\cal A},\Upsilon\big)\in \overline{K}$, then there exists $\big({\cal A'},\Upsilon'\big)\in \overline{K}$ which is a decomposable–Henkin model formed from some family $\{ {\cal A'}_{i}\colon i\in I'\}\subseteq K$ by an ultrafilter $\cal F'$ and an isomorphism
\[
f: \big({\cal A},\Upsilon\big)\longrightarrow \big({\cal A'},\Upsilon'\big)
\]
such that $f$ induces a bijection between $\Upsilon$ and $\Upsilon'$.  (That is, for every $R\in\Upsilon$ we have $f^{*}(R)\in\Upsilon'$, and for every $S\in\Upsilon'$ we have $(f^{-1})^{*}(S)\in\Upsilon$.)

\end{remark}
\begin{definition}
    Fix the signature $\tau$, and let $\langle {\cal A}_{n} : n\in\omega\rangle$ be an $\omega$-long ultrachain with limit ${\cal A}_{\omega}$, such that
\[
{\cal A}_{n+1}={}^{I_{n}}{\cal A}_{n}/{\cal F}_{n},
\]
and $\big( {\cal A}_{n+1},\Upsilon_{n+1}\big)$ is a decomposable–Henkin model formed from the family $\{ {\cal B}_{j}\colon j\in J\}$ by the ultrafilter ${\cal G}\times{\cal F}_{1}\times\cdots\times{\cal F}_{n}$, where
\[
{\cal A}_{0}=\prod_{j\in J} {\cal B}_{j}/{\cal G}.
\]
Then $\big({\cal A}_{\omega},\Upsilon_{\mathrm{lim}}\big)$ is a limit–Henkin model,  where $\Upsilon_{\mathrm{lim}}$ is the set of  decomposable relations with respect to the family $\{ {\cal B}_{j}\colon j\in J\}$.

Often,  we say ``the limit–Henkin model $\big({\cal A}_{\omega},\Upsilon_{\mathrm{lim}}\big)$ formed from the family $\{ {\cal B}_{j}\colon j\in J\}$ by the ultrachain $\langle {\cal A}_{n} : n\in\omega\rangle$.''
\end{definition}

The following important definition concerns the Henkin semantics for decomposable-Henkin models. This is a generalization of the usual Tarski-style definition of truth.

\begin{definition}\label{def:henkin sem}(Henkin semantics)
    \item Fix a $\tau$ signature and let $L$ be the first-order language over $\tau$, let  $\big({\cal A},\Upsilon\big)$ be a decomposable-Henkin model and $\phi(R_{0},R_{1},...,R_{n-1})$ is a second-order formula where only occurs $R_{0},R_{1}...R_{n-1}$  as relation variables. Consider  $l:\{R_{0},R_{1}...R_{n-1}\}\rightarrow\Upsilon$  a function   which takes any $k$-ary relation variable into a $k$-ary relation, then:
    \begin{itemize}
        \item if $\phi(R_{0},R_{1}...R_{n-1})\in\Delta_{0\{L\cup\{R_{0},R_{1}...R_{n-1}\}\}}$ then $\big({\cal A},\Upsilon\big)\models_{\epsilon}\phi(R_{0},R_{1}...R_{n-1})[l]$, iff $$\langle{\cal A},l(R_{0}),l(R_{1})...l(R_{n-1})\rangle\models\phi$$ in the first-order way;
        \item if $\phi\equiv\exists R_{0}\psi$ then  $\big({\cal A},\Upsilon\big)\models_{\epsilon}\phi(R_{0},R_{1}...R_{n-1})[l]$, iff there is an $$l':\{R_{0},R_{1}...R_{n-1}\}\rightarrow\Upsilon$$ function ( which takes any $k$-ary relation variable into a $k$-ary relation) such $l'\overset{R_{0}}{\equiv}l$ (they differ at most on \(R_0\)) and $\big({\cal A},\Upsilon\big)\models_{\epsilon}\phi(R_{0},R_{1}...R_{n-1})[l']$;
        \item  if $\phi\equiv\forall R_{0}\psi$ then  $\big({\cal A},\Upsilon\big)\models_{\epsilon}\phi(R_{0},R_{1}...R_{n-1})[l]$, iff for all  $$l':\{R_{0},R_{1}...R_{n-1}\}\rightarrow\Upsilon$$ function (which takes any $k$-ary relation variable into a $k$-ary relation) such $l'\overset{R_{0}}{\equiv}l$ then $\big({\cal A},\Upsilon\big)\models_{\epsilon}\phi(R_{0},R_{1}...R_{n-1})[l']$.
        \end{itemize}
     The formula $\phi(R_{0},R_{1}...R_{n-1})$ is defined to be true in  $\big({\cal A},\Upsilon\big)$ ( write as $\big({\cal A},\Upsilon\big)\models_{\epsilon}\phi$) if for all  $l:\{R_{0},R_{1}...R_{n-1}\}\rightarrow\Upsilon$ function ( which takes any $k$-ary relation variable into a $k$-ary relation) $\big({\cal A},\Upsilon\big)\models_{\epsilon}\phi[l]$.

\end{definition}

\begin{remark}\label{rem: henkin inf}
    The previous definiton The previous definition extends in the obvious way to limit-Henkin models, and  if $\big( {\cal A}, \Upsilon_{\mathrm{lim}}\big)$ is a limit Henkin model, and $\phi\in SO$ is true in $\big( {\cal A}, \Upsilon_{\mathrm{lim}}\big)$, then we  write $\big( {\cal A}, \Upsilon_{\mathrm{lim}}\big)\models_{\lambda}\phi.$
\end{remark}
 We extend the usual $\mathrm{Th}$ and $\mathrm{Mod}$ operations as follows.
\begin{definition}
    Fix the signature $\tau$ and let $\cal A$, $\cal A'$ are $\tau$ structures and $\big({\cal B}, \Upsilon\big)$, $\big({\cal B'}, \Upsilon'\big)$ are decomposable-Henkin models, then:
\begin{itemize}
\item $\mathrm{Th}({\cal A})=\{\phi\in FO\text{ : }{\cal A}\models\phi\}$ and $\mathrm{Th}(\big({\cal B}, \Upsilon\big))=\{\phi\in FO\text{ : }\big({\cal B}, \Upsilon\big)\models_{\epsilon}\phi\}$;
\item if $\Gamma\subseteq SO$ is an arbitrary fragment then we denote $\mathrm{Th}_{\Gamma}({\cal A})=\{\gamma\in \Gamma\text{ : }{\cal A}\models\gamma\}$ and $\mathrm{Th}_{\Gamma}(\big({\cal B}, \Upsilon\big))=\{\gamma\in \Gamma\text{ : }\big({\cal B}, \Upsilon\big)\models_{\epsilon}\gamma\},$
\item if $\{\gamma\in \Gamma\text{ : }{\cal A}\models\gamma\}\subseteq \{\gamma'\in \Gamma\text{ : }{\cal A'}\models\gamma'\}$ then we write $\mathrm{Th}_{\Gamma}({\cal A})\leq \mathrm{Th}_{\Gamma}({\cal A'})$, and similarly define $\mathrm{Th}_{\Gamma}(\big({\cal B}, \Upsilon\big))\leq \mathrm{Th}_{\Gamma}(\big({\cal B'}, \Upsilon'\big))$.
\end{itemize}

If $\Sigma\subseteq SO$ and $\Lambda\subseteq L_{\kappa,\lambda,\mu}^{2}$ then
\begin{itemize}
\item by $\mathrm{Mod}(\Sigma)$ and $\mathrm{Mod}(\Lambda)$ we denote the classes of all $\tau$ structures in which $\Sigma$ and $\Lambda$, respectively, are true;
\item by $\mathrm{Mod}_{\epsilon}(\Sigma)$ we denote the class of all decomposable-Henkin models in which $\Sigma$ is true;
\item by $\mathrm{Mod}_{<\omega}(\Sigma)$ we denote the class of all \emph{ finite} $\tau$ structures in which $\Sigma$ is true.
\end{itemize}
\end{definition}
\begin{remark}
    If $\phi$ is some formula, then for simplicity, instead of $\mathrm{Mod}(\{\phi\})$ and its variants, we simply write $\mathrm{Mod}(\phi)$.
\end{remark}

\begin{lemma}\label{lem: Łos} Let $\tau$ be a signature,  $\cal F$ is an ultrafilter over $I$, and
     $\phi\in SO(\tau)$ is a second-order formula. Consider for $j\in J$ a decomposable-Henkin model $\big({\cal A}_{j},\Upsilon_{j}\big)$ formed from the family $\{ {\cal A}_{i,j}\text{ : } i\in I\}$ of $\tau$ structures by the ultrafilter ${\cal F}$,  and the decomposable-Henkin model  $\big( {\cal A}, \Upsilon\big)$ formed from the family $\{ {\cal A}_{i,j}\text{ : } i\in I,\text{ }j\in J \}$ by the ultrafilter ${\cal F}\times {\cal G}$. Then the following are equivalent:\begin{itemize}
    \item[(1)]  $\{ j\in J\text{ : } \big({\cal A}_{j},\Upsilon_{j}\big)\models_{\epsilon}\phi  \}\in {\cal G }$;
    \item[(2)] $\big( {\cal A}, \Upsilon\big)\models_{\epsilon}\phi$.
\end{itemize}

\end{lemma}

\begin{proof}
By induction (on the complexities of the formulas) we show that for any $\tau$ signature  and any $\phi\in SO(\tau)$, (1) and (2) are equivalent. Thus let $\tau$ be a signature.

   First assume that   $\phi\in FO(\tau)$, then the statement follows from Łoś's lemma.
Now let $\phi\in SO(\tau)$,
\[
\phi\equiv Q_{0}R_{0}\;Q_{1}R_{1}\;\dots\;Q_{n-1}R_{n-1}\;\psi,
\]
where each $Q_{i}$ is a second-order quantifier, $R_{0},\dots,R_{n-1}$ are second-order (relation) variables, and $\psi\in\Delta_{0}(\tau)$. Suppose that for all $\tau$ signature
$(1)$ and
$(2)$ are equivalent for formulas with complexities
less then $n$.

Note that it suffices to prove that (1) and (2) are equivalent in the case when $Q_{0}$ is existential. Hence assume that $Q_{0}$ is existential and let $\phi'= Q_{1}R_{1}...Q_{n-1}R_{n-1}
\psi$ such  $\phi = \exists R_{0} \phi'$ and  suppose that $\big( {\cal A},\Upsilon \big)\models_{\epsilon}\phi$. Then there exists a relation $R_{0}\in\Upsilon$
 such    $\big(\langle {\cal A},R_{0} \rangle,\Upsilon\big)\models_{\epsilon}\phi'$. Further since $R_{0}\in \Upsilon$ , there are relations $R_{0}^{j}\in\Upsilon_{j}$ such  $$R_{0} =\prod_{j\in J}R_{0}^{j} /{\cal G}.$$  By the induction assumption $$\{j\in J\text{ : }\big(\langle {\cal A}_{j},R_{0}^{j} \rangle,\Upsilon_{j}\big)\models_{\epsilon}\phi'\}\in{\cal G} $$ therefore since $R_{0}^{j}\in\Upsilon_{j}$ we have that  $$\{j\in J\text{ : }\big({\cal A}_{j},\Upsilon\big)\models_{\epsilon}\phi\}\in{\cal G}. $$
 Now assume that $$\{j\in J\text{ : }\big({\cal A}_{j},\Upsilon_{j}\big)\models_{\epsilon}\phi\}\in{\cal G}, $$ then there are $S_{0}^{j}\in\Upsilon_{j}$ such that $$\{j\in J\text{ : }\big(\langle{\cal A}_{j},S_{0}^{j} \rangle,\Upsilon_{j}\big)\models_{\epsilon}\phi'\}\in{\cal G}. $$ Consider $$S_{0} =\prod_{j\in J}S_{0}^{j} /{\cal G}$$ then by the induction assumption $\big(\langle {\cal A},S_{0} \rangle,\Upsilon\big)\models_{\epsilon}\phi'$ hence $S_{0}\in\Upsilon$, we have  $\big({\cal A},\Upsilon\big)\models_{\epsilon}\phi.$ This
completes the induction.
\end{proof}

The previous lemma generalizes to limit–Henkin models as well.

\begin{lemma}\label{lem: Łos-limit}

Let $\tau$ be a signature,  $\phi\in SO(\tau)$,  and $\big({\cal A}_{\omega},\Upsilon_{\mathrm{lim}}\big)$ a limit-Henkin model formed from the family $\{ {\cal B}_{j}\colon j\in J\}$ of $\tau$ structures by the ultrachain $\langle {\cal A}_{n} : n\in\omega\rangle$, where
\[
{\cal A}_{0}=\prod_{j\in J} {\cal B}_{j}/{\cal G}.
\] Then the following are equivalent:\begin{itemize}
    \item[(1)]$\{j\in J\text{ : } {\cal B}_{j}\models\phi\}\in {\cal G}$;
    \item[(2)] $\big({\cal A}_{\omega},\Upsilon_{\mathrm{lim}}\big)\models_{\lambda}\phi.$
\end{itemize}

\end{lemma}
\begin{proof}
   By induction (on the complexities of the formulas) we show that for any $\tau$ signature  and any $\phi\in SO(\tau)$, (1) and (2) are equivalent.

   Thus let $\tau$ be a signature and let $\big({\cal A}_{\omega},\Upsilon_{\mathrm{lim}}\big)$ be a limit-Henkin model formed from the family $\{ {\cal B}_{j}\colon j\in J\}$ of $\tau$ structures by the ultrachain $\langle {\cal A}_{n} : n\in\omega\rangle$.

   First assume that   $\phi\in FO(\tau)$.  The statement follows from Łoś's Lemma together with the fact that ${\cal A}_{0}$ is an elementary submodel of ${\cal A}_{\omega}$.

Now suppose that $\phi\in SO(\tau)$ and
\[
\phi\equiv Q_{0}R_{0}\;Q_{1}R_{1}\;\dots\;Q_{l-1}R_{l-1}\;\psi,
\]
where each $Q_{i}$ is a second-order quantifier, $R_{0},\dots,R_{l-1}$ are second-order (relation) variables, and $\psi\in\Delta_{0}$.  Furthermore, assume that for every $\tau$-signature the statement holds for formulas of complexity less than $l$.  By Lemma~\ref{lem: Łos} it follows that for every $\xi\in\Delta_{l-1}$ we have
\[
\big({\cal A}_{\omega},\Upsilon_{\mathrm{lim}}\big)\models_{\lambda}\xi
\quad\text{iff}\quad
\big({\cal A}_{n},\Upsilon_{n}\big)\models_{\epsilon}\xi
\ \text{ for all }n\in\omega.
\]

As noted in Lemma~\ref{lem: Łos}, it is again sufficient to consider the case that $Q_{0}$ is existential.
 Hence assume that $Q_{0}$ is existential
and write $\phi' = Q_{1}R_{1}\dots Q_{l-1}R_{l-1}\,\psi$, so that $\phi=\exists R_{0}\,\phi'$.  Suppose
\[
\big({\cal A}_{\omega},\Upsilon_{\mathrm{lim}}\big)\models_{\lambda}\phi,
\]
then there exists a relation $R_{0}\in\Upsilon_{\mathrm{lim}}$ such that
\[
\big(\langle {\cal A}_{\omega},R_{0}\rangle,\Upsilon\big)\models_{\lambda}\phi'.
\]
Since $R_{0}\in\Upsilon_{\mathrm{lim}}$, there is some $m\in\omega$ and relations $R_{0}^{\mathcal{A}_{m}}\in\Upsilon_{m}$ with the property that $\langle{\cal A}_{\omega},R_{0}\rangle$ is the limit of the ultrachain
\[
\big\langle \langle{\cal A}_{k},R_{0}^{{\cal A}_{k}}\rangle : m<k<\omega\big\rangle,
\]
where for each $k>m$
\[
\langle{\cal A}_{k+1},R_{0}^{{\cal A}_{k+1}}\rangle
= {}^{I_{k}}\langle{\cal A}_{k},R_{0}^{{\cal A}_{k}}\rangle/{\cal F}_{k}.
\]
By the induction hypothesis we obtain $\big(\langle{\cal A}_{m},R_{0}^{{\cal A}_{m}}\rangle, \Upsilon_{m}\big)\models_{\epsilon}\phi'$, therefore $$\big({\cal A}_{m}, \Upsilon_{m}\big)\models_{\epsilon}\phi$$ and hence by Lemma~\ref{lem: Łos} it follows that
\[
\{j\in J : {\cal B}_{j}\models\phi\}\in{\cal G}.
\]

Now suppose that $\{j\in J\text{ : } {\cal B}_{j}\models\phi\}\in {\cal G}$,  then there are relations $S_{0}^{j} \subseteq {}^{k} B_{j}$ for some $0<k$ such that
$$\{j\in J\text{ : } \langle{\cal B}_{j},S_{0}^{j}\rangle\models\phi'\}\in {\cal G}$$
 where $\phi' = Q_{1}R_{1}\dots Q_{l-1}R_{l-1}\,\psi$, and let
$$\langle{\cal A}_{0},S_{0}^{{\cal A}_{0}}\rangle=\prod_{j\in J} \langle{\cal B}_{j},S_{0}^{j}\rangle/{\cal G}.$$
Moreover, for every $n\in\omega$ there is $S_{0}^{\mathcal{A}_{n}}\in\Upsilon_{n}$ and $S_{0}^{\mathcal{A}_{\omega}}\in\Upsilon_{\mathrm{lim}}$ relations, so that
\[
\big\langle \langle {\cal A}_{n},S_{0}^{{\cal A}_{n}}\rangle : n\in \omega \big\rangle
\]
is an ultrachain whose limit is $\langle {\cal A}_{\omega},S_{0}^{\mathcal{A}_{\omega}}\rangle$.  Then
$\big(\langle {\cal A}_{\omega}, S_{0}^{\mathcal{A}_{\omega}}\rangle, \Upsilon_{\mathrm{lim}}\big)$
is a limit-Henkin model formed from the family $\{\langle{\cal B}_{j},S_{0}^{j}\rangle\colon j\in J\}$ by the ultrachain $\big\langle \langle{\cal A}_{n},S_{0}^{{\cal A}_{n}}\rangle : n\in \omega\big\rangle.$
Then, by the induction hypothesis, since
\[
\{\,j\in J\colon \langle{\cal B}_{j},S_{0}^{j}\rangle\models\phi'\,\}\in {\cal G},
\]
we obtain
\[
\big(\langle {\cal A}_{\omega}, S_{0}^{\mathcal{A}_{\omega}}\rangle, \Upsilon_{\mathrm{lim}}\big)\models_{\lambda}\phi',
\] therefore
\[
\big( {\cal A}_{\omega}, \Upsilon_{\mathrm{lim}}\big)\models_{\lambda}\phi.
\]
This completes the proof of the lemma.

\end{proof}
\subsection{Spaces of formulas}\label{spf}
\label{spf}
As further preparation, we survey topological spaces defined on certain sets of formulas.
The notion of these spaces is very similar to the well-known Stone space of first-order types: each such space is a zero-dimensional Hausdorff space.
We will see that the closed, open, and clopen subsets of these spaces stand in a straightforward correspondence with classes of structures that can be axiomatized by the formulas determining the space.

\begin{definition}\label{def: toptér}
Fix a $\tau$ signature and let \( \Gamma\) be a fragment  of second- order formulas, which is closed under Boolean operations (for example \(\Gamma=\Delta_n\) or \(\Gamma=SO\) ect...), and suppose \(|\Gamma|=\kappa\). Fix a well-ordering  $\Gamma=\{\gamma _i : {i\in \kappa}\}$, and let $$C(\Gamma)=\{x\in 2^{\kappa} \text{ : there is a model } {\cal A}, \text{ such  } \gamma_{i}\in Th_{\Gamma}({\cal A}) \text{ if and only } x(i)=1 \} $$ a subspace of the $\kappa$-Cantor set $2^{\kappa}$ with a product topology.
 If $i\in\kappa$ then $$N_{i}=\{x\in C(\Gamma) : x(i)=1\}$$
is a basic clopen set. At this point we note that, when we consider $C(\Gamma)$ as the space associated to $\Gamma$, we always associate to $\Gamma$ the ordering fixed above.

 Note that, since the class $\Gamma$ is closed under Boolean operations, the set $$\{ N_{i} : i\in \kappa\}$$ form a basis for $C(\Gamma)$.
If $\Gamma$ is countable then we fix the ultra-metric $d$ on $C(\Gamma)$: for $x,y\in C(\Gamma)$ $$d(x,y)=\mathrm{max}\{ 2^{-i} : x(i)\neq y(i)\}.$$

\end{definition}

The following  lemma clarify the relationship between classes axiomatizable by formulas from \(\Gamma\) and the closed subsets of \(\mathcal{C}(\Gamma)\).

\begin{lemma}\label{lem: zárt}
   Fix a signature $\tau$.
    A class $K$ of $\tau$-structures can be axiomatizable by a set of  $\Gamma$-formulas if and only if the set $$T_{K}=\{x\in 2^{\kappa} \text{ : there is a model } {\cal A}\in K, \text{ such  } \gamma_{i}\in Th_{\Gamma}({\cal A}) \text{ if and only } x(i)=1 \} $$ is a closed subset of $C(\Gamma)$.

\end{lemma}

\begin{proof}
All of this follows from the fact that \(\{N_i : i\in\kappa\}\) is a basis for \(C(\Gamma)\).

    Let $K$ axiomatizable by a set of  $\Gamma$-formulas $\Sigma$, then $$\bigcap_{\phi_{i}\in\Sigma}N_i=T_K$$ which is closed. Conversely assume that $T_K$ is a closed subset of $C(\Gamma)$. Since \(\{N_i : i\in\kappa\}\) is a basis, there is a $J\subseteq\kappa$, such $$T_{K}=\bigcap_{j\in J}N_j.$$ Then $K=\mathrm{Mod}(\{\gamma_{j}: j\in J\}).$

\end{proof}

\begin{remark}\label{rem:metric}
     Fix a signature $\tau$ and assume that $\Gamma$ is countable. Let a class  $\tau$-structures $K$, then $K$ is  axiomatizable by a single $\Gamma$-formula if and only if, $d(T_{K},T_{K^c})>0$.
\end{remark}

Consider now $C(FO)$, and for simplicity assume that the signature $\tau$ is countable, so $FO=\{\phi_{i}\text{ : } i\in \omega\}$. Note that $C(FO)$ is compact: indeed, let $(x_{n})_{n\in \omega}\subset C(FO)$ be a Cauchy sequence. By the definition of $C(FO)$ there are $\tau$-structures ${\cal A}_n$ such that $\phi_{i}\in \mathrm{Th}({\cal A}_{n}) \text{ if and only if } x_{n}(i)=1$. Since $(x_{n})_{n\in \omega}\subset C(FO)$ is a Cauchy sequence, for each $i$ either $$\{ n\in\omega\text{ : } \phi_{i}\in \mathrm{Th}({\cal A}_{n})\}$$ or $$\{ n\in\omega\text{ : } \neg\phi_{i}\in \mathrm{Th}({\cal A}_{n})\}$$ is finite. Let ${\cal F}$ be a regular ultrafilter over $\omega$ and let ${\cal A}= \prod_{n \in \omega} {\cal A}_{n}/{\cal F}$; then $\mathrm{lim}x_{n}= x$ where $x(i)=1$ if and only if $\phi_{i}\in Th({\cal A})$. Therefore $C(FO)$ is closed in $2^{\omega}$, and hence compact.

In Theorem~\ref{thm: NP} we will give a second-order formula $\psi$ which expresses that a structure is infinite, and let
$$\phi_{n}\equiv=\exists x_{0}... \exists x_{n-1}\bigwedge_{i, j\in n,\; i\neq j}x_{i}\neq x_{j},$$
be the formula that expresses that the universe of a structure has $n$ elements, and let $\Sigma=\{\neg \psi, \phi_{n}\textrm{ : } 0<n\}$.

Consider $\Gamma$ the Boolean closure of $\Sigma$. Note that for every $\sigma\in[\Sigma]^{<\omega}$ there is a structure ${\cal A}_{\sigma}$ in which ${\cal A}_{\sigma}\models \bigwedge \sigma$.  However,  there is a convergent subsequence of  $\{T_{\{{\cal A}_{\sigma}\}}\text{ : }\sigma\in[\Sigma]^{<\omega}\}$ that has no limit in $C(\Gamma)$. Hence $C(\Gamma)$ is not closed in $2^{\kappa}$.

The previous two examples show that the compactness of $C(\Gamma)$ depends on whether the compactness theorem holds for the given set of formulas.

We will see that the compactification of $C(\Gamma)$ plays a central role in the understanding of finitely axiomatizable classes.

\begin{definition}
    Let $\Gamma$ be a fragment of second-order formulas which closed under Boolean operations and $|\Gamma|=\kappa$. Then $\overline{C(\Gamma)}$ denotes the topological closure of $C(\Gamma)$ in $2^{\kappa}$, let \begin{equation*}
\widetilde{C(\Gamma)}=\left\{\,x\in 2^{\kappa}\text{ : }
\begin{array}{@{}l@{}}
\text{there is a decomonable-Henkin model } \big({\cal A},\Upsilon\big),\\[2pt]
\text{such  } \gamma_{i}\in Th_{\Gamma}(\big({\cal A},\Upsilon\big)) \text{ if and only } x(i)=1
\end{array}
\right\},
\end{equation*}
and if $K$ be a class of $\tau$-structures, then we denote  \begin{equation*}
T_{\overline{K}}=\left\{\,x\in 2^{\kappa}\text{ : }
\begin{array}{@{}l@{}}
\text{there is a decomonable-Henkin model } \big({\cal A},\Upsilon\big)\in \overline{K},\\[2pt]
\text{such  } \gamma_{i}\in Th_{\Gamma}(\big({\cal A},\Upsilon\big)) \text{ if and only } x(i)=1
\end{array}
\right\}.
\end{equation*}

\end{definition}
\begin{remark}
We denote the basis of $\widetilde{C(\Gamma)}$ by $\{ N_{i}\text{ : }i\in I\}$. The definition in Definition~\ref{def: toptér} can be carried over to this setting as well.

\end{remark}
\begin{lemma}\label{lem: zárthenk}
    Let $\Gamma$ be as before, then we have $\overline{C(\Gamma)}=\widetilde{C(\Gamma)}.$ Moreover if  $K$ is a class of $\tau$-structures, then $\overline{T_{K}}=T_{\overline{K}}$, where $\overline{T_{K}}$ is the closer of $T_{K}$ in $2^{\kappa}.$
\end{lemma}

\begin{proof}
    First we show that $\overline{C(\Gamma)}\subseteq\widetilde{C(\Gamma)}$. For this, it suffices to prove that: for all $x\in 2^{\kappa}$ such for every basic open set $U\subseteq 2^{\kappa}$, we have  $x\in U\implies U\cap C(\Gamma)\neq\emptyset$, then $x\in\widetilde{C(\Gamma)}$. Pick $x\in 2^{\kappa}$ such for every basic open set $U\subseteq 2^{\kappa}$, if  $x\in U$ then $ U\cap C(\Gamma)\neq\emptyset$, and consider the set $$\Sigma=\{\gamma_{i}\text{ : } x(i)=1\}\cup\{(\neg\gamma_{i})\text{ : } x(i)=0\}.$$ Then  the hypothesis implies,  for every $j\in [\Sigma]^{<\omega}$ there is a $\tau$-structure ${\cal A}_{j}$ such that ${\cal A}_{j}\models\bigwedge j$. Note that the family $\{j^{\partial}\text{ :} j\in [\Sigma]^{<\omega}\}$ has the finite intersection property, and hence extends to an ultrafilter ${\cal F}$ over $[\Sigma]^{<\omega}$, where $$j^{\partial}=\{j'\in [\Sigma]^{<\omega}\text{ : } j\subseteq j'\}.$$ Let \(\big({\cal A},\Upsilon\big)\) be the decomposable–Henkin model  formed from the family $$\{ {\cal A}_{j}\text{ : } j\in [\Sigma]^{<\omega}\}$$ by the ultrafilter $\cal F$, then  by Lemma~\ref{lem: Łos}, we have $\Sigma=\mathrm{Th}_{\Gamma}((\big({\cal A},\Upsilon\big))$ therefore $x\in\widetilde{C(\Gamma)}.$

    Now we show that $\widetilde{C(\Gamma)}\setminus\overline{C(\Gamma)}=\emptyset$. Suppose indirectly that there is $x\in\widetilde{C(\Gamma)}\setminus\overline{C(\Gamma)}$. Then on the one hand there exist a decomposable-Henkin model $\big({\cal B},\Upsilon\big)$ formed from the family $\{ {\cal B}_{j}\text{ : } j\in j\in J\}$ by the ultrafilter $\cal G$ with $\big({\cal B},\Upsilon\big)\models_{\epsilon}\gamma_{i}$ iff $x(i)=1$ for all $i\in \kappa$. On the other hand since $x\notin\overline{C(\Gamma)}$, there is a basic open set $U\subset 2^{\omega}$ with $x\in U$ but $U\cap C(\Gamma)=\emptyset$, hence there are $I,I'\in[\kappa]^{<\omega}$ such that $$\big({\cal B},\Upsilon\big)\models_{\epsilon}\big(\bigwedge_{i\in I} \gamma_{i}\big)\land\big(\bigwedge_{i'\in I'}(\neg \gamma_{i'})\big)$$ while $$\mathrm{Mod}(\big(\bigwedge_{i\in I} \gamma_{i}\big)\land\big(\bigwedge_{i'\in I'}(\neg \gamma_{i'})\big))=\emptyset.$$ This is a contradiction, since by Lemma~\ref{lem: Łos} $$\emptyset\neq\{j\in J \text{ : } {\cal B}_{j}\models\big(\bigwedge_{i\in I} \gamma_{i}\big)\land\big(\bigwedge_{i'\in I'}(\neg \gamma_{i'})\big)\}\in \cal G.$$

The moreover part of the lemma can be proved in the same way.
 Therefore we proved the lemma.
\end{proof}

\begin{lemma}\label{lem: 4eqv}
     Let $\Gamma$ be as before and let $K$ be a class of $\tau$-structures, then the following are equivalent: \begin{itemize}\label{lem: TK}
         \item[(1)] there is a formula $\gamma\in \Gamma,$ such $K=\mathrm{Mod}(\gamma)$;\item[(2)]  there  is a formula $\gamma\in \Gamma,$ such $\overline{K}=\mathrm{Mod}_{\epsilon}(\gamma)$ ;
         \item[(3)] both $\overline{K}$ and $\overline{K^{c}}$ can be axiomatizable by
a set of $\Gamma$-formulas (that is, there are $\Lambda_{0},\Lambda_{1}\subseteq \Gamma,$ such $\overline{K}=\mathrm{Mod}_{\epsilon}(\Lambda_{0})$ and $\overline{K^{c}}=\mathrm{Mod}_{\epsilon}(\Lambda_{1})$);\item[(4)]  $T_{\overline{K}}\cap T_{\overline{K^{c}}}=\emptyset$.
     \end{itemize}
\end{lemma}
\begin{proof}
     Note that for any decomposable-Henkin model \(\big({\cal A},\Upsilon\big)\) formed from the family
$\{ {\cal A}_{i}\text{ : } i\in I\}$
by the ultrafilter $\cal F$ we have
$$I_{0}=\{ {\cal A}_{i}\in K \text{ : } i\in I\} \quad\quad I_{1}=\{ {\cal A}_{i}\in K^{c} \text{ : } i\in I\}$$
and either $I_{0}\in {\cal F}$ or $I_{1}\in{\cal F}$. Hence for every decomposable-Henkin model \(\big({\cal A},\Upsilon\big)\) either \(\big({\cal A},\Upsilon\big)\in \overline{K}\) or \(\big({\cal A},\Upsilon\big)\in \overline{K^{c}}\) (of course only one can hold). So we have $\overline{K^{c}}= \overline{K}^{c}$. From this we obtain that (2) and (1) are equivalent using Lemma~\ref{lem: Łos}. (2) implies (3) and (3) implies (4) obviously.

Now assume that (4) holds, so $T_{\overline{K}}$ and $T_{\overline{K^{c}}}$ are disjoint  sets in $\widetilde{C(\Gamma)}$. Then, by Lemma~\ref{lem: zárthenk}, $T_{\overline{K}}$ and $T_{\overline{K^{c}}}$ are closed in $\widetilde{C(\Gamma)}$ and by  the previous remark, $T_{\overline{K}}\cup T_{\overline{K^{c}}}=\widetilde{C(\Gamma)}.$ Since $T_{\overline{K^{c}}}$ is closed, there is $J\subseteq \kappa$ such that
$$T_{\overline{K^{c}}}=\bigcap_{j\in J} (N_{j})^{c},$$
 hence $\{N_{j}\text{ : }j\in J\}$ is an open cover of $T_{\overline{K}}$. Because $\widetilde{C(\Gamma)}$ is compact   and $T_{\overline{K}}$ is closed in $\widetilde{C(\Gamma)}$, we can choose $J_{0}\in[J]^{<\omega}$ such that
$$T_{\overline{K}}=\bigcup_{j\in {J_{0}}}N_{j}$$
and thus $$\overline{K}=\mathrm{Mod}(\bigvee_{j\in J_{0}}\gamma_{j}).$$ Therefore, (2)  holds, which proves  the lemma.

\end{proof}
As a final preparatory step, we recall the definition of the second-order Stone space, and establish a lemma showing that the ultraproduct can be regarded as a closure operator on certain spaces of formulas, similarly to what we saw in Lemma~\ref{lem: zárthenk}.

\begin{definition}
    Fix a signature $\tau$, an infinite cardinal $\kappa$, and consider the set $T$ of $\kappa$-2-types, moreover let $s=(s_{i})_{i\in\kappa}$ and $F$ be as in Definition~\ref{def: type}. Assume $|F|=\lambda$ and fix a well-ordering $\{\phi_{j}\text{ : }j\in\lambda\}=F$. Then $T$ naturally inherits the topology from $2^{\lambda}$, generated by the clopen sets
\[
N_{i}=\{p\in T\text{ : } \phi_{i}\in p\},
\]
and denote this topological space by $S^{2}_{\kappa}(\tau)$.

Let $\cal A$ be a $\tau$-structure, and let $$\overrightarrow{R_{\kappa}}=\{R_i : i\in\kappa\}$$ be a $\kappa$-long sequence of relations on $\cal A$, such that for every $i\in\kappa$ the relation $R_{i}$ is $s_{i}$-ary (we also say that the relations $\overrightarrow{R_{\kappa}}$ are of kind $s$). Then the $\overrightarrow{R_{\kappa}}$-realized $\kappa$-2-type is
$$
\mathbf{tp}({\cal A},\overrightarrow{R_{\kappa}})=\{\phi\in F\text{ : }\langle {\cal A}, R_{i}\rangle_{i\in \kappa}\models\phi\}\in S^{2}_{\kappa}(\tau).
$$

\end{definition}
\begin{remark}\label{rem: tp=Th}
    Note that the sets $\mathbf{tp}({\cal A},\overrightarrow{R_{\kappa}})$ and $\mathrm{Th}(\langle {\cal A}, R_{i}\rangle_{i\in \kappa})$ differ only in that $\mathbf{tp}({\cal A},\overrightarrow{R_{\kappa}})$ contains only closed formulas. Since every first-order formula is equivalent to its universal closure, the set $\mathbf{tp}({\cal A},\overrightarrow{R_{\kappa}})$ uniquely determines $\mathrm{Th}(\langle {\cal A}, R_{i}\rangle_{i\in \kappa})$.

\end{remark}

\begin{lemma}\label{lem: ult-lezár}

Let $S^{2}_{\kappa}(\tau)$ be as above, and  let  $\{{\cal A}_{i}\text{ : }i\in I\}$ is a set of $\tau$-structures,  $\overrightarrow{R^{i}_{\kappa}}$ are relations of kind $s$ on ${\cal A}_{i}$ and  set $$U=\{\mathbf{tp}\big({\cal A}_{i},\overrightarrow{R^{i}_{\kappa}}\big)\in S^{2}_{\kappa}(\tau)\text{ : } i\in I \}.$$ Consider  the ultraproduct $${\cal A}=\prod_{j\in J}{\cal A}_{j}/{\cal F}$$ where $\cal F$ is an ultrafilter over some set $J\subseteq I$ and $\overrightarrow{R_{\kappa}}$ are relations of kind $s$ on ${\cal A}$ with $$R_{k}=\prod_{j\in J}R_{k}^{j}/{\cal F}\quad k\in \kappa$$ then $\mathbf{tp}\big({\cal A},\overrightarrow{R_{\kappa}}\big)\in\overline{U},$ here $\overline{U}$ denotes the topological closure of $U$ in $2^{\lambda}$.

Furthermore if  $p\in \overline{U},$ then there exist an ultrafilter $\cal G$ over over some set $L\subseteq I$ relations $\overrightarrow{S^{l}_{\kappa}}$ of kind $s$ on ${\cal A}_{l},$ such $\mathbf{tp}\big({\cal B},\overrightarrow{S_{\kappa}}\big)=p,$ where  $${\cal B}=\prod_{l\in L}{\cal A}_{j}/{\cal G}\quad\text{and}\quad S_{k}=\prod_{l\in L}S_{k}^{l}/{\cal G}\quad k\in \kappa.$$

Therefore taking ultraproducts can be regarded as a kind of closure operator on the space of realizable $2$-types.

\end{lemma}
The proof of this lemma is essentially the same as that of Lemma~\ref{lem: zárthenk}, hence we omit the details.

\section{Pseudo-elementary classes}\label{sec: pseudo}

In this section we consider an algebraic characterization of pseudo-elementary classes. First recall the definition of pseudo-elementary classes.
\begin{definition}
A class $K$ of $\tau$-structures is a \emph{basic pseudo-elementary class} ($K\in PC$)  if there exists a signature $\tau'\supseteq\tau$ and a first-order $\tau'$-sentence $\phi\in FO(\tau')$ such that
\[
K=\{{\cal A} \text{ : } \text{there is a $\tau'$-structure }{\cal A}'\text{ expanding }{\cal A}\text{ with }{\cal A}'\models\phi\}.
\]
We say that $K$ is a \emph{pseudo-elementary class}  ($K\in PC_{\Delta}$) if there exists a  signature $\tau'\supseteq\tau$ and a first-order $\tau'$-Theory $\Gamma\subseteq FO(\tau')$ such that
\[
K=\{{\cal A} \text{ : } \text{there is a $\tau'$-structure }{\cal A}'\text{ expanding }{\cal A}\text{ with }{\cal A}'\models\Gamma\}
\] where the   language expansion is $\Delta$ (that is, $|FO(\tau')\setminus FO(\tau)|=\Delta$).

\end{definition}

\begin{remark}
    When we speak about $PC_{\Delta}$-classes, for simplicity we will always assume that $\Delta$ is infinite. However, it is easy to see that our statements can be easily reformulated for finite $\Delta$ as well.

\end{remark}
\subsection{PC-classes}
Pseudo-elementary classes are clearly related to classes axiomatizable by certain second-order formulas. We first deal with PC-classes. Observe that a class $K$ of $\tau$-structures is PC if and only if there exists a formula $\phi \in \Sigma_{1}^{1}$ such that $K = \mathrm{Mod}(\phi)$ (see e.g. exercise 4.1.17 \cite{chk}). Therefore, the algebraic characterization of PC-classes is equivalent to the characterization of classes definable by a $\Sigma_{1}^{1}$ formula. 

The notion of inseparability was introduced by Sági in \cite{cl1}, and here we generalize this definition, further generalizations will be given in Section~\ref{sec: second-order}.

\begin{definition}\label{def: pseudo-insep}
    Fix a signature $\tau $ and let $\{ {\cal A}_{i}\text{ : }i\in I\}$ and $\{{\cal B}_{j}\text{ : }j\in J\}$ be two sets of $\tau$-structures. Then  $\{ {\cal A}_{i}\text{ : }i\in I\}$ and $\{{\cal B}_{j}\text{ : }j\in J\}$ are defined to be
\emph{inseparable} from each other iff there are ultraproducts $$ {\cal A}=\prod_{i\in I}{\cal A}_{i}/{\cal F},\quad \quad {\cal B}=\prod_{j\in J}{\cal B}_{j}/{\cal G},$$ such that\begin{itemize} 
    \item[(i)]  for every finite system
of  decomposable relations $R_{0},...,R_{n-1} $ of $\cal A$,
there exist ultrapowers $${\cal A}' = {}^{I'}{\cal A} / {\cal
F'},  \quad\quad {\cal B'}={}^{J'}{\cal B} / {\cal G'}$$ and  an isomorphism
$f: {\cal A'} \rightarrow {\cal B'}$ such that $f^{*}(R_{0}^{\cal
A'}),...,f^{*}(R_{n-1}^{\cal A'})$ are decomposable in ${\cal B'},$ where $$\langle {\cal A'}, R_{0}^{\cal A'},...,R_{n-1}^{\cal A'} \rangle = {}^{I'}\langle {\cal A}, R_{0},...,R_{n-1} \rangle / {\cal F'},$$ here decomposability is understood with respect to the  ultrafilter ${\cal G}\times {\cal G'}$;\item[(ii)]  for every finite system
of   decomposable relations $S_{0},...,S_{m-1} $ of $\cal B$,
there exist ultrapowers $${\cal B''} = {}^{J''}{\cal B} / {\cal
G''},  \quad\quad {\cal A''}={}^{I''}{\cal A} / {\cal F''}$$ and  an isomorphism
$g: {\cal B''} \rightarrow {\cal A''}$ such that $g^{*}(S_{0}^{\cal
B''}),...,g^{*}(S_{m-1}^{\cal B''})$ are decomposable in ${\cal A''},$ where 
$$\langle {\cal B''}, S_{0}^{\cal B''},...,S_{m-1}^{\cal B''} \rangle = {}^{J''}\langle {\cal B}, S_{0},...,S_{m-1} \rangle / {\cal G''},$$  here decomposability is understood with respect to the  ultrafilter ${\cal F}\times {\cal F''}$. 
\end{itemize} 
The set  $\{{\cal B}_{j}\text{ : }j\in J\}$  is defined to be
inseparable from  $\{ {\cal A}_{i}\text{ : }i\in I\}$, if we require only (i).

\end{definition}
\begin{definition}\label{def: pseudo-pikk}
   Let $\tau$ be a signature and $K$ a class of $\tau$-structures. We say that $K$ is  \emph{closed under inseparability}, if for any set of $\tau$-structures $\{ {\cal A}_{i}\text{ : }i\in I\}\subseteq K$ and $\{{\cal B}_{j}\text{ : }j\in J\}$, if $\{ {\cal A}_{i}\text{ : }i\in I\}$ and $\{{\cal B}_{j}\text{ : }j\in J\}$ are inseparable from each other, then
\[
K\cap \{{\cal B}_{j}\text{ : }j\in J\}\neq \emptyset.
\]

We say that $K$ is \emph{strongly}  closed under inseparability, if for any set of $\tau$-structures $\{{\cal B}_{j}\text{ : }j\in J\}$ that is inseparable from some set $\{ {\cal A}_{i}\text{ : }i\in I\}\subseteq K$, we have
\[
K\cap \{{\cal B}_{j}\text{ : }j\in J\}\neq \emptyset.
\]

\end{definition}

Note that the two definitions concern purely internal properties.
Before turning to the discussion of Theorem~\ref{thm: PC}, we state a classical result of first-order model theory.

\begin{lemma}
\label{lem: separacios}(Folklore\footnote{See corollary 6.1.17 of \cite{chk}})
    Fix a signature $\tau$, let $K$ and $L$ be two disjoint classes of $\tau$-structures, such that both  are closed under ultraproducts. Then there is a first-order formula $\phi\in FO$, such $K\models\phi$ and $L\models\neg\phi.$
\end{lemma}

\begin{theorem}\label{thm: PC}
    Fix a signature $\tau$ and  let $K$ be a class of $\tau$-structures, then the following are equivalent:\begin{itemize}
        \item[(1)] $K\in PC;$
        \item[(2)] $K$ is definable by a second-order existential formula (that is, there exists a formula $\phi\in \Sigma_{1}^{1}$, such $K=\mathrm{Mod}(\phi)$);
        \item[(3)] $K$ is strongly closed under inseparability;
        \item[(4)] $K$ is closed under ultraproducts and inseparability.
    \end{itemize}
\end{theorem}
\begin{proof}
It is easy to see that (1) and (2) are equivalent, see also exercise 4.1.17 \cite{chk}.

Now assume that (2) holds. Thus there exists a formula $\phi\in \Sigma_{1}^{1}$ such that $K=\mathrm{Mod}(\phi)$, where $\phi\equiv\exists R_{0},\dots,\exists R_{n-1}\psi$ and $\psi\in\Delta_{0}$. Moreover, we may assume that each relation variable is $k$-ary. Suppose that $\{ {\cal B}_{j}\text{ : }j\in J\}$ is a set of $\tau$-structures which is inseparable from the set $\{ {\cal A}_{i}\text{ : }i\in I\}\subseteq K$. Then there are ultraproducts
\[
{\cal A}=\prod_{i\in I}{\cal A}_{i}/{\cal F},\qquad
{\cal B}=\prod_{j\in J}{\cal B}_{j}/{\cal G},
\]
such that condition (i) of Definition~\ref{def: pseudo-insep} holds. Since ${\cal A}_{i}\models\phi$, there exist relations $R^{i}_{0},\dots,R^{i}_{n-1}\subseteq {}^{k} A_{i}$ with
\[
\langle {\cal A}_{i}, R^{i}_{0},\dots,R^{i}_{n-1}\rangle\models\psi,
\]  for each $i\in I$ (without loss of generality, we may assume that each relation variable is $k$-ary in $\phi$).
Hence
\[
\langle {\cal A}, R_{0},\dots,R_{n-1}\rangle
=\prod_{i\in I} \langle {\cal A}_{i}, R^{i}_{0},\dots,R^{i}_{n-1}\rangle/{\cal F},
\]
so $\langle {\cal A}, R_{0},\dots,R_{n-1}\rangle\models\psi$ by Łoś's Lemma, and moreover $R_{0},\dots,R_{n-1}\subseteq{}^{k} A$ are decomposable relations on ${\cal A}$. Therefore there exist ultrapowers ${\cal A}' = {}^{I'}{\cal A} / {\cal F'}, \quad {\cal B}'={}^{J'}{\cal B} / {\cal G'}$
and an isomorphism $f: {\cal A}' \rightarrow {\cal B}'$ such that $f^{*}(R_{0}^{{\cal A}'}),\dots,f^{*}(R_{n-1}^{{\cal A}'})$ are decomposable in ${\cal B}'$ where $$\langle {\cal A'}, R_{0}^{\cal A'},...,R_{n-1}^{\cal A'} \rangle ={}^{I'} \langle {\cal A}, R_{0},...,R_{n-1} \rangle / {\cal F'}.$$ Equivalently, there exist relations $S^{j,j'}_{0},\dots,S^{j,j'}_{n-1}\subseteq{}^{k} B_{j}$ such that
\[
\langle{\cal B}', f^{*}(R_{0}^{{\cal A}'}),\dots,f^{*}(R_{n-1}^{{\cal A}'})\rangle
\cong
\prod_{\langle j, j'\rangle\in J\times J'}\langle{\cal B}_{j},S^{j,j'}_{0},\dots,S^{j,j'}_{n-1}\rangle/({\cal G}\times {\cal G'}).
\] Then, by Łoś's Lemma, on the one hand
$$\langle {\cal A'}, R_{0}^{\cal A'},...,R_{n-1}^{\cal A'} \rangle\models\psi,$$
and hence we obtain
$$\langle{\cal B}', f^{*}(R_{0}^{{\cal A}'}),\dots,f^{*}(R_{n-1}^{{\cal A}'})\rangle\models\psi.$$
Again by Łoś's Lemma we get
$$\{\langle j, j'\rangle\in J\times J' \text{ : } \langle{\cal B}_{j},S^{j,j'}_{0},\dots,S^{j,j'}_{n-1}\rangle\models\psi\}\in {\cal G}\times{\cal G'},$$
hence
$$\emptyset\neq\{j\in J\text{ : } {\cal B}_{j}\models\phi\}\in {\cal G},$$
so
\[
K\cap \{{\cal B}_{j}\text{ : }j\in J\}\neq \emptyset.
\]
Therefore $K$  is strongly  closed under inseparability.

Now suppose that (3) holds, i.e. $K$  is strongly  closed under inseparability. Then clearly $K$  is  closed under inseparability. Let $\{ {\cal A}_{i}\text{ : }i\in I\}\subseteq K$ and let ${\cal F}$ be an ultrafilter over $I$, and set
\[
{\cal A}=\prod_{i\in I} {\cal A}_{i}/{\cal F}.
\]
Note that $\{ {\cal A}\}$ is inseparable from $\{ {\cal A}_{i}\text{ : }i\in I\}$, hence ${\cal A}\in K$ (since $K$ is strongly closed under inseparability). Therefore $K$ is closed under ultraproducts as well. 

Finally, we show that (4) implies (2). So assume (4). First we show that there is a $\Delta_{1}$-formula $\phi$ such that $K=\mathrm{Mod}(\phi)$. By Lemma~\ref{lem: TK} it is enough to prove that $T_{\overline{K}}\cap T_{\overline{K^{c}}}=\emptyset$. Assume indirectly that there is $x\in T_{\overline{K}}\cap T_{\overline{K^{c}}}$; that is, there is a decomposable–Henkin model \(\big({\cal A},\Upsilon_{0}\big)\) formed from the family \(\{ {\cal A}_{i}\text{ : } i\in I\}\subseteq K\) by the ultrafilter $\cal F$, and another decomposable–Henkin model \(\big({\cal B},\Upsilon_{1}\big)\) formed from the family \(\{ {\cal B}_{j}\text{ : } j\in J\}\subseteq K^{c}\) by the ultrafilter $\cal G$, such that
\[(*)\quad
\mathrm{Th}_{\Delta_{1}}\big( \big({\cal A},\Upsilon_{0}\big)\big)
=\mathrm{Th}_{\Delta_{1}}\big( \big({\cal B},\Upsilon_{1}\big)\big)
=\{\delta_{i}\in\Delta_{1}\text{ : } x(i)=1\}.
\]
 We claim that then \(\{ {\cal A}_{i}\text{ : } i\in I\}\) and \(\{ {\cal B}_{j}\text{ : } j\in J\}\) are inseparable from each other, which contradicts the assumption that $K$ is closed under inseparability.
 
 It suffices to show that \(\{ {\cal A}_{i}\text{ : } i\in I\}\) is inseparable from \(\{ {\cal B}_{j}\text{ : } j\in J\}\), since our assumptions are symmetric.
Let $R_{0},...,R_{n-1}\in \Upsilon_{0}$ be decomposable relations of $\cal A$. Set
$$\Lambda= \mathrm{Th}(\langle{\cal A}, R_{0},...,R_{n-1}\rangle).$$
Then, by $(*)$, for every $\lambda\in[\Lambda]^{<\omega}$ there exist decomposable relations $S^{\lambda}_{0},...,S^{\lambda}_{n-1}\in \Upsilon_{1}$  on $\cal B$ such that
$$\langle{\cal B}, S^{\lambda}_{0},...,S^{\lambda}_{n-1}\rangle\models\land \lambda.$$
The family $\{\lambda^{\partial}\text{ : } \lambda\in[\Lambda]^{<\omega}\}$ has the finite intersection property, and therefore extends to an ultrafilter ${\cal G'}$ over $[\Lambda]^{<\omega}$, where
$\lambda^{\partial}=\{ \lambda'\in[\Gamma]^{<\omega}\text{ : } \lambda\subseteq \lambda'\}.$
Consider the ultraproduct
$$\langle{\cal B'}, S_{0},...,S_{n-1}\rangle=\prod_{\lambda\in[\Lambda]^{<\omega}}\langle{\cal B}, S^{\lambda}_{0},...,S^{\lambda}_{n-1}\rangle/{\cal G'}.$$
Then, by Łoś's Lemma, $\mathrm{Th}(\langle{\cal B'}, S_{0},...,S_{n-1}\rangle)=\Lambda$, and hence by Shelah's isomorphism theorem \cite{Shelah} there are isomorphic ultrapowers
$${}^{I''}\langle{\cal A}, R_{0},...,R_{n-1}\rangle/{\cal G''}\cong{}^{I''}\langle{\cal B'}, S_{0},...,S_{n-1}\rangle/{\cal G''}=\langle{\cal B''}, S'_{0},...,S'_{n-1}\rangle.$$
Moreover, by Theorem~\ref{thm: fubini}, $S'_{0},...,S'_{n-1}$ are decomposable relations with respect to ${\cal G}\times{\cal G'}\times\cal G''$. 
Therefore we have shown that $\{ {\cal A}_{i}\text{ : } i\in I\}$ and $\{ {\cal B}_{j}\text{ : } j\in J\}$ are inseparable from each other, which is a contradiction. 

Now we show that $K$ is not only definable by a $\Delta_{1}$-formula but also by a $\Sigma_{1}^{1}$-formula.  Indeed, suppose $K=\mathrm{Mod}(\phi)$ for some $\phi\in\Delta_{1}$.  Note that any $\psi\in\Delta_{1}$ can be written as $\psi\equiv\psi_{0}\land\psi_{1}$ with $\psi_{0}\in\Sigma_{1}^{1}$ and $\psi_{1}\in\Pi_{1}^{1}$.  Hence write $\phi\equiv\phi_{0}\land\phi_{1}$ with $\phi_{0}\in\Sigma_{1}^{1}$ and $\phi_{1}\in\Pi_{1}^{1}$.  If $\mathrm{Mod}(\phi_{0})\neq K$ set
\[
L=\mathrm{Mod}(\phi_{0})\setminus K.
\]
Then $L=\mathrm{Mod}(\phi_{0}\lor\neg\phi_{1})$, and $\phi_{0}\lor\neg\phi_{1}\in\Sigma_{1}^{1}$.  It is easy to see that $L$ is closed under ultraproducts, and by hypothesis so is $K$, while $K\cap L=\emptyset$.  Therefore, by Lemma~\ref{lem: separacios} there is a first-order formula $\xi\in FO(\tau)$ with $K\models\xi$ and $L\models\neg\xi$.  But then
\[
K=\mathrm{Mod}(\phi_{0}\land\xi),
\]
and $\phi_{0}\land\xi\in\Sigma_{1}^{1}$.  This completes the proof.\end{proof}

At this point we turn to the study classes of finite structures. 

\begin{definition}
    A class \(K\) of finite \(\tau\)-structures is in \(\textsf{NP}\)  if there exists a polynomial \(p\) and a nondeterministic Turing machine \(M\) such that there is an accepting computation of $M$ on input \(\operatorname{enc}(\mathcal A)\) of length at most \(p(|\operatorname{enc}(\mathcal A)|)\) iff \(\mathcal A\in K\), where \(\operatorname{enc}(\mathcal A)\) denotes the code of $\cal A$.

\end{definition}

\begin{theorem}\label{thm: NP}
    Fix a finite signature $\tau,$ and  let $K$ be a class of finite $\tau$-structures, then the following are equivalent: \begin{itemize}
        \item[(1)] $K\in\textsf{NP}$;
        \item[(2)] there exists a formula $\phi\in \Sigma_{1}^{1}$, such that $K=\mathrm{Mod}_{<\omega}(\phi);$
        \item[(3)] for any set of finite $\tau$-structures $\{ {\cal A}_{i}\text{ : }i\in I\}\subseteq K$, if  the set of finite $\tau$-structures $\{ {\cal B}_{j}\text{ : }j\in J\}$ is inseparable from $\{ {\cal A}_{i}\text{ : }i\in I\}$, then $K\cap \{{\cal B}_{j}\text{ : }j\in J\}\neq \emptyset.$
    \end{itemize}
\end{theorem}

\begin{proof}
First, we know that (1) and (2) imply each other by Fagin's theorem \cite{Fagin}.

Second, it is easy to see that the class of all infinite $\tau$-structures is definable by a $ \Sigma_{1}^{1}$ formula. A standard example of such a formula is the following: a set is infinite if and only if it admits an order without top element; if $R$ is a binary relation variable, then this property can be expressed as
$$\exists R (\forall y \neg R y y\land\forall y \forall z \forall w (Ry z \land R z w\implies R yw)\land \forall y \forall z  (R y z \lor R zy \lor y =z)\land \forall y \exists z  R y z ).$$ Let $\psi$ denote the previous formula, and let $L=\mathrm{Mod}(\psi)$ be the class of infinite $\tau$-structures.  Suppose (2) holds for $K$, i.e.\ there exists a formula $\phi\in \Sigma_{1}^{1}$ such that $K=\mathrm{Mod}_{<\omega}(\phi)$. Then $L\cup K=\mathrm{Mod}(\psi\lor\phi)$.  By Theorem~\ref{thm: PC} we know that $L\cup K$ is strongly closed under inseparability.  Hence, for any sets of finite $\tau$-structures $\{ {\cal A}_{i}\text{ : }i\in I\}\subseteq K$, if the set of finite $\tau$-structures $\{ {\cal B}_{j}\text{ : }j\in J\}$ is inseparable from $\{ {\cal A}_{i}\text{ : }i\in I\}$, then 
\[
(L\cup K)\cap \{{\cal B}_{j}\text{ : }j\in J\}\neq \emptyset,
\]
and since the class of finite structures in $L\cup K$ is $K$, we obtain
\[
K\cap \{{\cal B}_{j}\text{ : }j\in J\}\neq \emptyset.
\]

Now assume (3) holds for $K$, then clearly $L\cup K$ is closed under inseparability.  Therefore, by Theorem~\ref{thm: PC} there exists $\phi\in\Sigma_{1}^{1}$ such that $L\cup K=\mathrm{Mod}(\psi)$, and hence $K=\mathrm{Mod}_{<\omega}(\psi)$.  Thus (2) holds, and the theorem is proved.
\end{proof}
Our previous result gives a purely algebraic characterization of when a class of structures is in \textsf{NP}. 

It is interesting that the statement "whenever an arbitrary set of finite structures 
$\{{\cal B}_{j} : j\in J\}$ is inseparable from a set of finite structures 
$\{{\cal A}_{i} : i\in I\}$, then $\{{\cal A}_{i} : i\in I\}$ is also inseparable 
from $\{{\cal B}_{j} : j\in J\}$",  implies $\mathsf{NP}=\mathsf{co\text{-}NP}$.
However, the converse remains open:

\begin{problem}
Does $\mathsf{NP}=\mathsf{co\text{-}NP}$ imply that whenever an arbitrary set of 
finite structures $\{{\cal B}_{j} : j\in J\}$ is inseparable from a set of finite 
structures $\{{\cal A}_{i} : i\in I\}$, then $\{{\cal A}_{i} : i\in I\}$ is also 
inseparable from $\{{\cal B}_{j} : j\in J\}$?
\end{problem}

Now consider the following example:\begin{example}\label{ex: np-példa}
    Let $H$ be the class of graphs having a Hamiltonian cycle, let ${\cal C}_{2n}$ denote the cyclic graph on $2n$ vertex, and ${\cal D}_{n}$ be the graph consisting of two disjoint $n$-cycles.  Let $\cal F$ be a regular ultrafilter over $\omega$, then
$$
{\cal C}=\prod_{n\in\omega}{\cal C}_{2n}/{\cal F}\quad\text{and}\quad{\cal D}=\prod_{n\in\omega}{\cal D}_{n}/{\cal F}
$$
are isomorphic:  it is easy to see that both graphs are isomorphic to a disjoint union of continuum many two-way infinite paths. We know $H\in\textsf{NP}$, and for every $n\in\omega$ we have ${\cal C}_{2n}\in H$ while ${\cal D}_{n}\notin H$. Hence by Theorem~\ref{thm: NP} the family $\{{\cal D}_{n}\text{ : } n\in \omega\}$ is not inseparable from $\{{\cal C}_{2n}\text{ : } n\in \omega\}$.  Thus there is no isomorphism between $\cal C$ and $\cal D$ which induces a bijection between the decomposable relations. This shows that the set of decomposable relations of an ultraproduct depend on the ultrafilter and also on the factors.

\end{example}
    \subsection{PC$_{\Delta}$-classes}
Now we turn to the characterization of $PC_{\Delta}$-classes. Throughout this subsection, we fix the infinite cardinal $\Delta$, and the sequence of integers $s=(s_{i})_{i\in\Delta}$, where every integer occur in infinitely many times as in Definition~\ref{def: type}.
 First, observe that a class $K$ of $\tau$-structures is pseudo-elementary if and only if there is $\Gamma\subset F$ such that
$$K= \mathrm{Mod}(\exists \overrightarrow{X_{\Delta}}\bigwedge_{\gamma\in\Gamma}\gamma)$$
where $F$ is the set of formulas defined in Definition~\ref{def: type} (with  $\kappa=\Delta$). We note that the above formula that defines the class $K$ is a second-order infinitary formula.  First, we deal with the omission of types.

\begin{definition}
    Let $\tau$ be a signature and $\cal A$ be a $\tau$-structure. We say that $\cal A$ omits the type $p\in S^{2}_{\Delta}(\tau)$ if there do not exist  relations $\overrightarrow{R_{\Delta}}$ of kind $(s_{i})_{i\in\Delta}$ on it, such that $\mathbf{tp}({\cal A},\overrightarrow{R_{\kappa}})=p$.
For a class $K$ of $\tau$-structures, we denote
$$\mathbf{O}( K)=\{p\in S^{2}_{\Delta}(\tau)\text{ : all elements of }K \text{ omits } p\}.$$
Moreover, we say that a class of $\tau$-structures $L$ is axiomatizable by types omissions  if there is a set $\Pi\subseteq S^{2}_{\Delta}(\tau)$ with $\Pi\subseteq\mathbf{O}( L)$ and for every structure ${\cal B}\in L^{c}$ there exists a type $p\in \Pi$ such that $\cal B$ realizes $p$.

\end{definition}

\begin{definition}
    Fix the signature $\tau$. Then the class $K$ of $\tau$-structures has property  \begin{itemize}
        \item[(A)] if $\mathcal{A}$ is a $\tau$-structure and for every tuple of relations $\overrightarrow{R_{\Delta}}$ of kind $(s_{i})_{i\in\Delta}$ on $\mathcal{A}$ there exist a structure $\mathcal{B}\in K$ and relations $\overrightarrow{S_{\Delta}}$ of kind $(s_{i})_{i\in\Delta}$ on $\mathcal{B}$ such that $\langle\mathcal{A},R_{i}\rangle_{i\in\Delta}$ and $\langle\mathcal{B},S_{i}\rangle_{i\in\Delta}$ has  isomorphic ultrapowers, then $\mathcal{A}\in K$.

    \end{itemize}
\end{definition}

The previous definition is purely algebraic. The following theorem uses it to characterize those classes that are axiomatizable by types omissions.

\begin{theorem}\label{thm: omitting}
Fix a signature $\tau$ and a class $K$ of $\tau$-structures. Then the following are equivalent: \begin{itemize}
    \item[(1)]$K$ is axiomatizable by types omissions;
    \item[(2)] $K$ has the property (A).
\end{itemize}
    
\end{theorem}
\begin{proof}
    First assume that $K$ is axiomatizable by types omissions, so there is a set $\Pi\subseteq S^{2}_{\Delta}(\tau)$ with $\Pi\subseteq\mathbf{O}( K)$ and for every structure ${\cal C}\in K^{c}$ there exists a type $p\in \Pi$ such that $\cal C$ realizes $p$. Let $\mathcal{A}$ be a $\tau$-structure, such that for every tuple of relations $\overrightarrow{R_{\Delta}}$ of kind $(s_{i})_{i\in\Delta}$ on $\mathcal{A}$ there exist a structure $\mathcal{B}\in K$ and relations $\overrightarrow{S_{\Delta}}$ of kind $(s_{i})_{i\in\Delta}$ on $\mathcal{B}$ such that $\langle\mathcal{A},R_{i}\rangle_{i\in\Delta}$ and $\langle\mathcal{B},S_{i}\rangle_{i\in\Delta}$ has isomorphic ultrapowers, so we have $\mathbf{tp}({\cal A},\overrightarrow{R_{\Delta}})=\mathbf{tp}({\cal B},\overrightarrow{S_{\Delta}})$. Thus there is no $p\in S^{2}_{\Delta}(\tau)$ that $\mathcal A$ realises while every structure in $K$ omits it, hence $\mathbf{O}(K)\subseteq\mathbf{O}(\{\mathcal A\})$, therefore $\mathcal A$ omits $\Pi$, and consequently ${\cal A}\in K$.  Therefore we obtain that $K$ has property (A).

Now assume that $K$ has property (A).  We claim that $\mathbf{O}(K)$ axiomatizes $K$.  Towards a contradiction, suppose there is ${\cal A}\in K^{c}$ with $\mathbf{O}(K)\subseteq\mathbf{O}(\{\cal A\})$.  Then for every tuple of relations $\overrightarrow{R_{\Delta}}$ of kinds $(s_{i})_{i\in\Delta}$ on $\mathcal A$ there exists a structure $\mathcal B$ and relations $\overrightarrow{S_{\Delta}}$ of kinds $(s_{i})_{i\in\Delta}$ on $\mathcal B$ such that $\langle\mathcal{A},R_{i}\rangle_{i\in\Delta}$ and $\langle\mathcal{B},S_{i}\rangle_{i\in\Delta}$ have isomorphic ultrapowers, (otherwise $\mathbf{tp}({\cal A},\overrightarrow{R_{\Delta}})\in\mathbf{O}(K)$ would follow, which is impossible). Since $K$ has property (A), ${\cal A}\in K$, which is a contradiction. This completes the proof.
\end{proof}

\begin{lemma}\label{lem: Pc-zárt}
    Fix a signature $\tau$ and   a class $K$ of $\tau$-structures. Then $K\in PC_{\Delta}$ if and only if $K^{c}$ is axiomatizable by a closed set of types omissions, that is, there exists a closed set $\Pi\subseteq S^{2}_{\Delta}(\tau)$  with $\Pi\subseteq\mathbf{O}( K^{c})$ and for every structure ${\cal B}\in K$ there exists a type $p\in \Pi$ such  $\cal B$ realizes $p$.
\end{lemma}

\begin{proof}
The proof is very similar to Lemma~\ref{lem: zárt}. Assume $K\in PC_{\Delta}$.  Then there is $\Gamma\subset F$ such that
$$K= \mathrm{Mod}(\exists \overrightarrow{X_{\Delta}}\bigwedge_{\gamma_{i}\in\Gamma}\gamma_{i}),$$
and hence
$$\Pi=\{p\in S^{2}_{\Delta}(\tau)\text{ : }\gamma\in p \text{ for all }\gamma_{i}\in\Gamma\}=\bigcap_{\gamma_{i}\in\Gamma} N_{i}$$
is closed in $S^{2}_{\Delta}(\tau)$.  Clearly $\Pi\subseteq\mathbf{O}( K^{c})$, since if there were ${\cal B}\in K^{c}$ and $p\in\Pi$ with ${\cal B}$ realizing $p$, then
$${\cal B}\in\mathrm{Mod}(\exists \overrightarrow{X_{\Delta}}\bigwedge_{\gamma_{i}\in\Gamma}\gamma_{i}),$$
which is impossible.

Now suppose $K^{c}$ is axiomatizable by a closed set of types omissions $\Pi$.  Therefore there is a set $\Sigma\subseteq F$ such that
$$\Pi=\bigcap_{\gamma_i\in \Sigma} N_{i}.$$
We claim that
$$K= \mathrm{Mod}(\exists \overrightarrow{X_{\Delta}}\bigwedge_{\gamma_{i}\in\Sigma}\gamma_{i}).$$
Indeed, if ${\cal A}\in K$ then there is $p\in\Pi$ that ${\cal A}$ realises, hence
$${\cal A}\in \mathrm{Mod}(\exists \overrightarrow{X_{\Delta}}\bigwedge_{\gamma_{i}\in\Sigma}\gamma_{i}),$$
and, for every ${\cal B}\in K^{c}$ we have
$${\cal B}\notin\mathrm{Mod}(\exists \overrightarrow{X_{\Delta}}\bigwedge_{\gamma_{i}\in\Sigma}\gamma_{i}).$$
\end{proof}

Consider the following purely algebraic   definitions.

\begin{definition}
    We call an operation $\mathbf{F}$ a \emph{$\Delta$-switch operation} if it assigns to each pair $\big( {\cal A},\overrightarrow{R_{\Delta}}\big)$ the relations $\overrightarrow{S_{\Delta}}$ $$\mathbf{F}\big( {\cal A},\overrightarrow{R_{\Delta}}\big)=\overrightarrow{S_{\Delta}},$$ where $\cal A$ is a $\tau$-structure and both $\overrightarrow{R_{\Delta}}$ and $\overrightarrow{S_{\Delta}}$ are relations of kind $s$ on $\cal A$.

\end{definition}
\begin{definition}\label{def: (B)}
   Fix a signature $\tau$. The class $K$ of $\tau$-structures has property

\begin{itemize}
\item[(B)] if there exists a $\Delta$-switch operation $\mathbf{F}$ such that for every set $I$ and every ultrafilter ${\cal F}$ over $I$ the following holds:

\noindent If ${\cal A}_{i}\in K$ and $\overrightarrow{R^{i}_{\Delta}}$ are relations of kind $s$ on $\mathcal{A}_{i}$, and there do \emph{not} exist structures ${\cal C}_{i}\in K^{c}$ with relations $\overrightarrow{S^{i}_{\Delta}}$ of kind $s$ on ${\cal C}_{i}$ such that, for all $i\in I$, the pairs $\langle \mathcal{A}_{i},R^{i}_{j}\rangle_{j\in \Delta}$ and $\langle \mathcal{C}_{i},S^{i}_{j}\rangle_{j\in \Delta}$ have isomorphic ultrapowers. Then if  $\mathcal{B}$  $\tau$-structure and $\overrightarrow{R_{\Delta}}$ is a tuple relations  of kind $s$ on  $\mathcal{B}$, such that
\[
\langle \mathcal{B},R_{j}\rangle_{j\in \Delta}
\quad\text{and}\quad
\prod_{i\in I} \langle {\cal A}_{i}, \mathbf{F}\big( {\cal A}_{i},\overrightarrow{R^{i}}\big)_{j}\rangle_{j\in \Delta}/{\cal F}
\]
have isomorphic ultrapowers, then ${\cal B}\in K$.
\end{itemize}

There are many (indeed, a proper class) of $\Delta$-switch operations. For example, an operation $\mathbf{G}$ that, for every structure $\cal A$, assigns to each $k$-ary relation $R \subseteq {}^{k}A$ a $k$-ary relation $\mathbf{G}(R) \subseteq {}^{k}A$ for each $k \in \omega$, naturally determines a $\Delta$-switch operation. As we shall see in the proof of  Theorem~\ref{thm: PC_DELTA}, a $\Delta$-switch operation satisfying property (B) is typically not unique, in general, there is a lot of freedom in how it can be chosen.

\end{definition}

\begin{remark}
    Defintion \ref{def: (B)} can be easily translated into the language of type omissions as follows.  $K$ has  property (B) if there exists a $\Delta$-switch operation $\mathbf{F}$ such that for every set $I$ and every ultrafilter ${\cal F}$ over $I$, the following holds: if ${\cal A}_{i}\in K$ and $\overrightarrow{R^{i}_{\Delta}}$ are relations of kind $s$ on $\mathcal{A}_{i}$ with
\[
\mathbf{tp}\big( \mathcal{A}_{i},\overrightarrow{R^{i}_{\Delta}}\big)\in \mathbf{O}(K^{c})\quad\text{for all }i\in I,
\]
then
\[
\mathbf{tp}\big(\mathcal{A},\overrightarrow{P_{\Delta}}\big)\in \mathbf{O}(K^{c}),
\]
where
\[
\langle\mathcal{A},P_{j}\rangle_{j\in \Delta}
=\prod_{i\in I} \langle {\cal A}_{i}, \mathbf{F}\big( {\cal A}_{i},\overrightarrow{R^{i}}\big)_{j}\rangle_{j\in \Delta}/{\cal F}.
\]

\end{remark}

Finally, we turn  to the final theorem of the section.
\begin{theorem}\label{thm: PC_DELTA}
    Fix a signature $\tau$ and a class $K$ of $\tau$-structures. Then $K\in PC_{\Delta}$ iff $K$  has  property (B) and $K^{c}$ has  property (A).

\end{theorem}

\begin{proof}
    First suppose that $K\in PC_{\Delta}$. Then, by Lemma~\ref{lem: Pc-zárt}, $K^{c}$ is axiomatizable by a closed set closed set of types omissions, that is, there exists a closed set $\Pi\subseteq S^{2}_{\Delta}(\tau)$ with $\Pi\subseteq\mathbf{O}(K^{c})$ and for every structure ${\cal A}\in K$ there exists a type $p\in \Pi$ such that ${\cal A}$ realizes $p$. Hence, by Theorem~\ref{thm: omitting}, $K^{c}$ has property (A). Note that since $K^{c}$ is axiomatizable by a set $\Pi\subseteq S^{2}_{\Delta}(\tau)$ of omitted types, then for every ${\cal A}\in K$ there is a tuple of relations $\overrightarrow{S_{\Delta}({\cal A})}$ of kind $s$ on $\mathcal{A}$ such that
\[
\mathbf{tp}\big(\mathcal{A},\overrightarrow{S_{\Delta}({\cal A})}\big)\in \Pi.
\]
Define the $\Delta$-switch operation $\mathbf{F}$ as follows: if ${\cal A}\in K$ and $\overrightarrow{R_{\Delta}}$ are relations of kind $s$ on $\mathcal{A}$, set
\[
\mathbf{F}\big(\mathcal{A},\overrightarrow{R_{\Delta}}\big)=\overrightarrow{S_{\Delta}({\cal A})},
\]
and if ${\cal B}\in K^{c}$ and $\overrightarrow{P_{\Delta}}$ are relations of kind $s$ on $\mathcal{B}$, set
\[
\mathbf{F}\big(\mathcal{B},\overrightarrow{P_{\Delta}}\big)=\overrightarrow{P_{\Delta}}.
\] Let ${\cal A}_{i}\in K$ and $\overrightarrow{R^{i}_{\Delta}}$ be relations of kind $s$ on $\mathcal{A}_{i}$ with
\[
\mathbf{tp}\big( \mathcal{A}_{i},\overrightarrow{R^{i}_{\Delta}}\big)\in \mathbf{O}(K^{c})\quad\text{for all }i\in I,
\]
let ${\cal F}$ be an ultrafilter over $I$, and set
\[
\langle\mathcal{A},P_{j}\rangle_{j\in \Delta}
=\prod_{i\in I} \langle {\cal A}_{i}, \mathbf{F}\big( {\cal A}_{i},\overrightarrow{R^{i}}\big)_{j}\rangle_{j\in\Delta}/{\cal F},
\]
furthermore consider the set
\[
U=\{\mathbf{tp}\big({\cal A}_{i},\mathbf{F}\big( {\cal A}_{i},\overrightarrow{R^{i}}\big)_{j}\rangle_{j\in \Delta}\big)\in S^{2}_{\Delta}(\tau)\text{ : } i\in I \}.
\]
By the construction of ${\mathbf{F}}$ we have $U\subseteq \Pi$, and by Lemma~\ref{lem: ult-lezár} we obtain $\mathbf{tp}\big({\cal A}, \overrightarrow{P}\big)\in \overline{U}$. Since $\Pi$ is closed, $\overline{U}\subseteq \Pi$. Hence $K$ has property (B).

Now suppose that $K$ has property (B) and $K^{c}$ has property (A). Then, by (A), $K^{c}$ is axiomatizable by a set of types omissions, for example $\mathbf{O}(K^{c})$. Furthermore, there exists a $\Delta$-switch operation $\mathbf{F}$ such that for every set $I$ and every ultrafilter ${\cal F}$ over $I$ the following holds: if ${\cal A}_{i}\in K$ and $\overrightarrow{R^{i}_{\Delta}}$ are relations of kind $s$ on $\mathcal{A}_{i}$ with
\[
\mathbf{tp}\big( \mathcal{A}_{i},\overrightarrow{R^{i}_{\Delta}}\big)\in \mathbf{O}(K^{c})\quad\text{for all }i\in I,
\]
then
\[
\mathbf{tp}\big(\mathcal{A},\overrightarrow{P_{\Delta}}\big)\in \mathbf{O}(K^{c})
\] where
\[
\langle\mathcal{A},P_{j}\rangle_{j\in \Delta}
=\prod_{i\in I} \langle {\cal A}_{i}, \mathbf{F}\big( {\cal A}_{i},\overrightarrow{R^{i}}\big)_{j}\rangle_{j\in \Delta}/{\cal F}.
\] 
Consider the set
\[
V=\big\{ \mathbf{tp}\big(\mathcal{A},\overrightarrow{S_{\Delta}}\big)\in S^{2}_{\Delta}(\tau)\ :\ \mathcal{A}\in K,\ \overrightarrow{S_{\Delta}}=\mathbf{F}\big( {\cal A},\overrightarrow{R}_{\Delta}\big)\text{ with } \mathbf{tp}\big( {\cal A},\overrightarrow{R}_{\Delta}\big)\in\mathbf{O}(K^{c})\big\}.
\] We claim that $K^{c}$ is axiomatizable by the closed set $\overline{V}$ of closed set of types omits (here $\overline{V}$ is the topological closure of $V$ in $S^{2}_{\Delta}(\tau)$). First note that $V\subseteq \mathbf{O}(K^{c})$: if $p\in V$ then there is ${\cal A}\in K$ and relations $\overrightarrow{R_{\Delta}}$ of kind $s$ on $\mathcal{A}$ such that
\[
\mathbf{tp}\big( \mathcal{A},\overrightarrow{R_{\Delta}}\big)\in \mathbf{O}(K^{c}),
\] with $$p=\mathbf{tp}\big(\mathcal{A},\mathbf{F}\big(\mathcal{A},\overrightarrow{R_{\Delta}}\big)\big).$$
Let ${\cal F}$ be the trivial ultrafilter on a singleton set. Then by condition (B) we have $p\in \mathbf{O}(K^{c})$, hence $V\subseteq \mathbf{O}(K^{c})$.
Moreover, applying Lemma~\ref{lem: ult-lezár} we obtain that for every $p\in\overline{V}$ there exist ${\cal A}_{i}\in K$ and relations $\overrightarrow{R^{i}_{\Delta}}$ of kind $s$ on $\mathcal{A}_{i}$ such that
\[
\mathbf{tp}\big( \mathcal{A}_{i},\mathbf{F}\big( {\cal A}_{i},\overrightarrow{R^{i}}\big)\big)\in V\quad\text{for all }i\in I,
\]
and
\[
p=\mathbf{tp}\big(\mathcal{A},\overrightarrow{P_{\Delta}}\big),
\]
where
\[
\langle\mathcal{A},P_{j}\rangle_{j\in \Delta}
=\prod_{i\in I} \langle {\cal A}_{i}, \mathbf{F}\big( {\cal A}_{i},\overrightarrow{R^{i}}\big)_{j}\rangle_{j\in \Delta}/{\cal F}.
\]
Therefore, by property (B) we have $p\in \mathbf{O}(K^{c})$, and hence $\overline{V}\subseteq \mathbf{O}(K^{c})$. Finally, let ${\cal A}\in K$. Since $K^{c}$ is axiomatizable by the set $\mathbf{O}(K^{c})$ of omitted types, there exist relations $\overrightarrow{R_{\Delta}}$ of kind $s$ on $\cal A$ such that
\[
\mathbf{tp}\big(\mathcal{A},\overrightarrow{R_{\Delta}}\big)\in \mathbf{O}(K^{c}).
\]
Put
\[
\overrightarrow{S_{\Delta}}= \mathbf{F} \big(\mathcal{A},\overrightarrow{R_{\Delta}}\big).
\]
Then
$\mathbf{tp}\big(\mathcal{A},\overrightarrow{S_{\Delta}}\big)\in V,
$
which completes the proof of the theorem.
\end{proof}

\section{Second-order classes}\label{sec: second-order}
In this section we deal with the problem of characterization of second-order axiomatizable classes. Throughout the remainder of the paper we fix the signature $\tau$.

\begin{definition}
    Let $\{\mathcal{D}_{i} : i\in I\}$ and $\{\mathcal{E}_{j} : j\in J\}$ be two sets of $\tau$-structures. We say that $\{\mathcal{D}_{i} : i\in I\}$ and $\{\mathcal{E}_{j} : j\in J\}$ are \emph{completely inseparable} from each other if there exist ultraproducts
\[
\mathcal{A}_{0}=\prod_{i\in I}\mathcal{D}_{i}/\mathcal{G}
\qquad\text{and}\qquad
\mathcal{B}_{0}=\prod_{j\in J}\mathcal{E}_{j}/\mathcal{G}',
\]
and $\omega$-long ultrachains $\langle \mathcal{A}_{n}:n\in\omega\rangle$, $\langle \mathcal{B}_{n}:n\in\omega\rangle$ starting from them with limits $\mathcal{A}_{\omega}$ and $\mathcal{B}_{\omega}$, together with an isomorphism
\[
f:\mathcal{A}_{\omega}\longrightarrow\mathcal{B}_{\omega}
\]
such that $f$ induces a bijection between the decomposable relations of $\mathcal{A}_{\omega}$ and those of $\mathcal{B}_{\omega}$, e.g. for every
$R\in\Upsilon_{\mathrm{lim}}$ we have $f^{*}(R)\in\Upsilon^{0}_{\mathrm{lim}}$, and for every
$S\in\Upsilon'_{\mathrm{lim}}$ we have $(f^{-1})^{*}(S)\in\Upsilon_{\mathrm{lim}}$ (where $\Upsilon_{\mathrm{lim}}$ is the set of  decomposable relations with respect to the family $\{ {\cal D}_{i}\colon i\in I\}$ and  $\Upsilon^{0}_{\mathrm{lim}}$ is the set of  decomposable relations with respect to the family $\{ {\cal E}_{j}\colon j\in J\}$).

\end{definition}
Note that the previous definition depends solely on the internal structure of the sets of structures. The following lemma will be crucial for our purposes.
\begin{lemma}\label{lem: főlemma}
Let $\{\mathcal{D}_{i} : i\in I\}$ and $\{\mathcal{E}_{j} : j\in J\}$ be two sets of $\tau$-structures. Then the following are equivalent:\begin{itemize}
    \item[(1)] $\{\mathcal{D}_{i} : i\in I\}$ and $\{\mathcal{E}_{j} : j\in J\}$ are completely inseparable from each other;
    \item[(2)] there exist ultrafilters ${\cal G}$ over $I$ and ${\cal G'}$ over $J$ such that
\[
\{\phi\in SO\text{ : }\{i\in I\text{ : }{\cal D}_{i}\models \phi\}\in {\cal G}\}
=
\{\psi\in SO\text{ : }\{j\in J\text{ : }{\cal E}_{j}\models \psi\}\in {\cal G'}\}.
\]

\end{itemize}
    
\end{lemma}

\begin{proof}
   
Suppose that $\{\mathcal{D}_{i} : i\in I\}$ and $\{\mathcal{E}_{j} : j\in J\}$ are completely inseparable from each other. That is, there exist ultraproducts
\[
\mathcal{A}_{0}=\prod_{i\in I}\mathcal{D}_{i}/\mathcal{G}
\qquad\text{and}\qquad
\mathcal{B}_{0}=\prod_{j\in J}\mathcal{E}_{j}/\mathcal{G}',
\]
and $\omega$-long ultrachains $\langle \mathcal{A}_{n}:n\in\omega\rangle$, $\langle \mathcal{B}_{n}:n\in\omega\rangle$ starting from them with limits $\mathcal{A}_{\omega}$ and $\mathcal{B}_{\omega}$, together with an isomorphism
\[
f:\mathcal{A}_{\omega}\longrightarrow\mathcal{B}_{\omega},
\]
such that $f$ induces a bijection between the decomposable relations of $\mathcal{A}_{\omega}$ and those of $\mathcal{B}_{\omega}$. We claim that the ultrafilters ${\cal G}$  and ${\cal G}'$  satisfy
\[
\{\phi\in SO\text{ : }\{i\in I\text{ : }{\cal D}_{i}\models \phi\}\in {\cal G}\}
=
\{\psi\in SO\text{ : }\{j\in J\text{ : }{\cal E}_{j}\models \psi\}\in {\cal G'}\}.
\]
Consider the limit–Henkin model $\big({\cal A}_{\omega},\Upsilon_{\mathrm{lim}}\big)$ formed from the family $\{ {\cal D}_{i}\colon i\in I\}$ by the ultrachain $\langle {\cal A}_{n} : n\in\omega\rangle$, and the limit–Henkin model $\big({\cal B}_{\omega},\Upsilon^{0}_{\mathrm{lim}}\big)$ formed from the family $\{ {\cal E}_{j}\colon j\in J\}$ by the ultrachain $\langle {\cal B}_{n} : n\in\omega\rangle$. By our hypothesis it follows that
\[
\mathrm{Th}_{SO}(\big({\cal A}_{\omega},\Upsilon_{\mathrm{lim}}\big))
=\mathrm{Th}_{SO}(\big({\cal B}_{\omega},\Upsilon^{0}_{\mathrm{lim}}\big)).
\]
Moreover, by Lemma~\ref{lem: Łos-limit} we have
\[
\{\phi\in SO\text{ : }\{i\in I\text{ : }{\cal D}_{i}\models \phi\}\in {\cal G}\}
=\mathrm{Th}_{SO}(\big({\cal A}_{\omega},\Upsilon_{\mathrm{lim}}\big))
\]
and
\[
\{\psi\in SO\text{ : }\{j\in J\text{ : }{\cal E}_{j}\models \psi\}\in {\cal G'}\}
=\mathrm{Th}_{SO}(\big({\cal B}_{\omega},\Upsilon^{0}_{\mathrm{lim}}\big)).
\]
Therefore the desired equality holds, and (2) is satisfied.

Now suppose that there exist ultrafilters ${\cal G}$ over $I$ and ${\cal G'}$ over $J$ such that
\[
\{\phi\in SO\ :\ \{i\in I : {\cal D}_{i}\models \phi\}\in {\cal G}\}
=
\{\psi\in SO\ :\ \{j\in J : {\cal E}_{j}\models \psi\}\in {\cal G'}\}.
\]
Let
\[
\mathcal{A}_{0}=\prod_{i\in I}\mathcal{D}_{i}/{\cal G}
\qquad\text{and}\qquad
\mathcal{B}_{0}=\prod_{j\in J}\mathcal{E}_{j}/{\cal G'}.
\]
Consider the  decomposable-Henkin models $\big(\mathcal{A}_{0},\Upsilon_{0}\big),$ and $\big(\mathcal{B}_{0},\Upsilon^{0}_{0}\big)$   formed from the families $\{ \mathcal{D}_{i}\text{ : }i\in I\}$ and $\{ \mathcal{E}_{j}\text{ : }j\in J\}$ by the ultrafilters $\cal G$ and $\cal G'$. By Lemma~\ref{lem: Łos} $\mathrm{Th}_{SO}\big({\cal A}_{0}, \Upsilon_{0}\big)=\mathrm{Th}_{SO}\big({{{{{\cal B}}}}_{0}}, \Upsilon^{0}_{0}\big)$. We will recursively construct $\omega$-long ultrachains $\langle \mathcal{A}_{n}:n\in\omega\rangle$ and $\langle \mathcal{B}_{n}:n\in\omega\rangle$ starting from $\mathcal{A}_{0}$ and $\mathcal{B}_{0}$, respectively, so that there is an isomorphism between their limits which induces a bijection between their decomposable relations.
First we  need the following technical claim.
\begin{claim}\label{claim: t}
   Let $\big(\mathcal{M},\Upsilon\big)$ and $\big(\mathcal{N},\Upsilon^{0}\big)$ be two decomposable–Henkin models formed from the families $\{\mathcal{M}_{\ell}:\ell\in L\}$ and $\{\mathcal{N}_{s}:s\in S\}$ by the ultrafilters $\mathcal{L}$ and $\mathcal{S}$, respectively, let $\Check{\mathcal{M}}=\big\langle\mathcal{M},R\big\rangle_{R\in\Upsilon}$ and suppose that
\[
\mathrm{Th}_{SO}\big(\mathcal{M},\Upsilon\big)=\mathrm{Th}_{SO}\big(\mathcal{N},\Upsilon^{0}\big).
\]

Then there exist a set $\Lambda$ and an ultrafilter ${\cal S'}$ over $\Lambda$, such that for every $\lambda\in\Lambda$ and $R\in\Upsilon$ there is a relation $R^{\lambda}\in\Upsilon^{0}$, and there exist ultraproducts
\[
\Check{\mathcal{M}}'={}^{L'}\Check{\mathcal{M}}/{\mathcal{L}'}\qquad\text{and}\qquad
\Check{\mathcal{N}}'=\prod_{\lambda\in\Lambda}\big\langle\mathcal{N}',R^{\lambda}\big\rangle_{R\in\Upsilon}/{\cal S'},
\]
and an isomorphism
\[
h:\Check{\mathcal{M}}'\longrightarrow\Check{\mathcal{N}}',
\]
such that
\[
\mathrm{Th}_{SO}\big(\Check{\mathcal{M}}',\Upsilon'\big)
=\mathrm{Th}_{SO}\big(\Check{\mathcal{N}}',\Upsilon'^{0}\big).
\]

Where $\big(\mathcal{M}',\Upsilon'\big)$ and $\big(\mathcal{N}',\Upsilon'^{0}\big)$ are decomposable–Henkin models  formed from the families $\{\mathcal{M}_{\ell}:\ell\in L\}$ and $\{\mathcal{N}_{s}:s\in S\}$ by the ultrafilters $\mathcal{L}\times\mathcal{L}'$ and $\mathcal{S}\times\mathcal{S}'$, respectively, where $\mathcal{M}'$ and $\mathcal{N}'$ are the $\tau$-reducts of $\Check{\mathcal{M}}'$ and $\Check{\mathcal{N}}'$. Note that $(\Check{\mathcal{M}}',\Upsilon')$ and $(\Check{\mathcal{N}}',\Upsilon'^{0})$ are also decomposable–Henkin models formed from the families $\{\Check{\mathcal{M}}_{\ell}:\ell\in L\}$ and $\{\Check{\mathcal{N}}_{s}:s\in S\}$ by the ultrafilters $\mathcal{L}\times\mathcal{L}'$ and $\mathcal{S}\times\mathcal{S'}$, respectively, where for each $\ell\in L$ and $s\in S$ the structures  $\Check{\mathcal{M}}_{\ell}$ and $\Check{\mathcal{N}}_{s}$ are suitable expansions of ${\mathcal{M}}_{\ell}$ and ${\mathcal{N}}_{s}$, respectively.

\end{claim}\item  \emph{Proof.} Let $\big(\mathcal{M},\Upsilon\big)$ and $\big(\mathcal{N},\Upsilon^{0}\big)$ and $\Check{\mathcal{M}}$ be as in the above claim, and set 
\[
\Lambda'=[\mathrm{Th}_{SO}\big(\Check{\mathcal{M}},\Upsilon\big)]^{<\omega}.
\]
Since
\[
\mathrm{Th}_{SO}\big({\mathcal{M}},\Upsilon\big)=\mathrm{Th}_{SO}\big(\mathcal{N},\Upsilon^{0}\big),
\]
it follows that for every $\lambda\in\Lambda'$ and  $R\in\Upsilon$ there is a relation $R^{\lambda}\in\Upsilon^{0}$ such that
\[
\big(\langle {\cal N}, R^{\lambda}\rangle_{R\in\Upsilon},\Upsilon^{0}\big)\models\bigwedge\lambda.
\]
Let ${\cal S}''$ be an ultrafilter over $\Lambda'$ containing each set
\[
\lambda^{\partial}=\{\mu\in\Lambda'\ :\ \lambda\subseteq\mu\}\qquad(\lambda\in\Lambda').
\]
Put
\[
\Check{{\cal N}}''=\prod_{\lambda\in\Lambda'}\big\langle {\cal N}, R^{\lambda}\big\rangle_{R\in\Upsilon}/{\cal S}''.
\]
Then $\Check{{\cal N}}''$ and $\Check{{\cal M}}$ are elementarily equivalent (in the first-order sense), so there exist ultrapowers
\[
\Check{{\cal M}}'={}^{L'}\Check{{\cal M}}/{\cal L'}\quad\text{and}\quad
\Check{{\cal N}}'={}^{L'}\Check{{\cal N}}''/{\cal L'},
\]
and an isomorphism
\[
h:\Check{{\cal M}}'\longrightarrow\Check{{\cal N}}'
\] by \cite{Shelah}. Let $\mathcal{M}'$ and $\mathcal{N}'$ be the $\tau$-reducts of $\Check{\mathcal{M}}'$ and $\Check{\mathcal{N}}'$, respectively, and let $(\mathcal{M}',\Upsilon')$ and $(\mathcal{N}',\Upsilon'^{0})$ be the decomposable–Henkin models formed from the families $\{\mathcal{M}_{\ell}:\ell\in L\}$ and $\{\mathcal{N}_{s}:s\in S\}$ by the ultrafilters $\mathcal{L}\times\mathcal{L}'$ and $\mathcal{S}\times\mathcal{S'}$, respectively.

We claim that
$
\mathrm{Th}_{SO}\big(\Check{\mathcal{M}}',\Upsilon'\big)
=\mathrm{Th}_{SO}\big(\Check{\mathcal{N}}',\Upsilon'^{0}\big)
$ holds.
Indeed, suppose that
$
\phi\in\mathrm{Th}_{SO}\big(\Check{\mathcal{M}}',\Upsilon'\big),$
then, by Lemma~\ref{lem: Łos},
$
\phi\in\mathrm{Th}_{SO}\big(\Check{\mathcal{M}},\Upsilon\big),
$
hence
\[
\{\lambda\in\Lambda' : \phi\in\lambda\}\in {\cal S}'',
\]
and therefore
\[
\{\lambda\in\Lambda' : \big(\langle\mathcal{N}, R^{\lambda}\rangle_{R\in \Upsilon},\Upsilon^{0}\big)\models\phi\}\in{\cal S}'',
\]
again by Lemma~\ref{lem: Łos} it follows that
$
\phi\in\mathrm{Th}_{SO}\big(\Check{\mathcal{N}}',\Upsilon'^{0}\big).
$
This proves the claim.

\item 
Now we construct the ultrachains $\langle \mathcal{A}_{n}:n\in\omega\rangle$ and $\langle \mathcal{B}_{n}:n\in\omega\rangle$ starting from $\mathcal{A}_{0}$ and $\mathcal{B}_{0}$, respectively, and the auxiliary ultrachains $\langle \mathcal{A'}_{n}:n\in\omega\rangle$ and $\langle \mathcal{B'}_{n}:0<n<\omega\rangle$, as follows. (Throughout, if $\cal C$ is a structure then $\delta_{\cal C}$ or $\delta'_{\cal C}$ denotes the diagonal embedding of $\cal C$ into a suitable ultrapower; whenever convenient we may assume $\delta'_{\cal C}=\delta_{\cal C}=\mathrm{Id}$.)

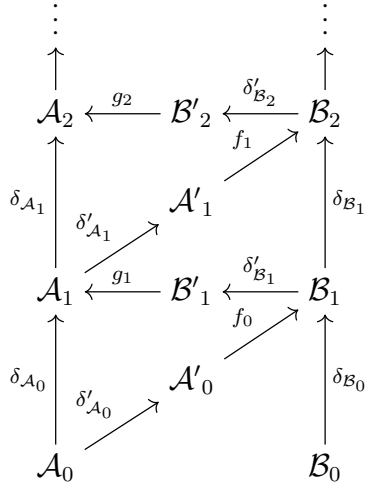
\[\begin{tikzcd}[ampersand replacement=\&,cramped,column sep=small,row sep=scriptsize]
	{\rotatebox{90}{\dots}} \&\&\&\& {\rotatebox{90}{\dots}} \\
	{{\cal A}_{2}} \&\& {{\cal B'}_{2}} \&\& {{\cal B}_{2}} \\
	\&\& {{\cal A'}_{1}} \\
	{{\cal A}_{1}} \&\& {{\cal B'}_{1}} \&\& {{\cal B}_{1}} \\
	\&\& {{\cal A'}_{0}} \\
	{{\cal A}_{0}} \&\&\&\& {{\cal B}_{0}}
	\arrow[from=2-1, to=1-1]
	\arrow["{g_{2}}"', from=2-3, to=2-1]
	\arrow[from=2-5, to=1-5]
	\arrow["{\delta'_{{\cal B}_{2}}}"', from=2-5, to=2-3]
	\arrow["{f_{1}}", from=3-3, to=2-5]
	\arrow["{\delta_{{\cal A}_{1}}}", from=4-1, to=2-1]
	\arrow["{\delta'_{{\cal A}_{1}}}", from=4-1, to=3-3]
	\arrow["{g_{1}}"', from=4-3, to=4-1]
	\arrow["{\delta_{{\cal B}_{1}}}"', from=4-5, to=2-5]
	\arrow["{\delta'_{{\cal B}_{1}}}"', from=4-5, to=4-3]
	\arrow["{f_{0}}", from=5-3, to=4-5]
	\arrow["{\delta_{{\cal A}_{0}}}", from=6-1, to=4-1]
	\arrow["{\delta'_{{\cal A}_{0}}}", from=6-1, to=5-3]
	\arrow["{\delta_{{\cal B}_{0}}}"', from=6-5, to=4-5]
\end{tikzcd}\]

\begin{itemize}
\item [(1a)]  Let $\Check{{\cal A}}_{0}=\big\langle{\cal A}_{0},R\big\rangle_{R\in \Upsilon}.$
\item[(1b)] Let $\Check{\cal B}_{0}={\cal B}_{0}.$
  \item[(2a)] For every $n\in\omega$ let $\Check{\Check{{\cal A}}}_{n+1}={}^{I_{n}}\Check{{\cal A}}_{n}/{\cal F}_{n}$ be an ultrapower, and let $\big( {\cal A}_{n+1},\Upsilon_{n+1}\big)$ be the decomposable–Henkin model formed from the family $\{ {\cal D}_{i}\colon i\in I\}$ by the ultrafilter ${\cal G}\times{\cal F}_{1}\times\cdots\times{\cal F}_{n}$, where ${\cal A}_{n+1}$ is the $\tau$ reduct of $\Check{\Check{{\cal A}}}_{n+1}$.
  \item[(2b)] For every $n\in\omega$ let $\Check{\Check{{\cal B}}}_{n+1}={}^{J_{n}}\Check{{\cal B}}_{n}/{\cal H}_{n}$ be an ultrapower, and let $\big( {\cal B}_{n+1},\Upsilon^{0}_{n+1}\big)$ be the decomposable–Henkin model formed from the family $\{ {\cal E}_{j}\colon j\in J\}$ by the ultrafilter ${\cal G'}\times{\cal H}_{1}\times\cdots\times{\cal H}_{n}$, where ${\cal B}_{n+1}$ is the $\tau$ reduct of $\Check{\Check{{\cal B}}}_{n+1}$.
 
  \item[(3a)] For every $n\in\omega$ let $\Check{{\cal A'}}_{n}={}^{I'_{n}}\Check{{\cal A}}_{n}/{\cal F'}_{n}$ be an ultrapower and let $\big( {\cal A'}_{n},\Upsilon'_{n}\big)$ be the decomposable–Henkin model formed from the family $\{ {\cal D}_{i}\colon i\in I\}$ by the ultrafilter ${\cal G}\times{\cal F}_{1}\times\cdots\times{\cal F}_{n-1}\times{\cal F'}_{n}$, where ${\cal A'}_{n} $ is the $\tau$ reduct of $\Check{{\cal A'}}_{n}$ and put ${\cal A}_{n+1}={}^{I''_{n}}{\cal A'}_{n}/{\cal F''}_{n}$ so that ${\cal F}_{n}={\cal F'}_{n}\times {\cal F''}_{n}$.
  \item[(3b)] For every $0<n<\omega$ let $\Check{{\cal B'}}_{n}={}^{J'_{n}}\Check{{\cal B}}_{n}/{\cal H'}_{n}$ be an ultrapower and let $\big( {\cal B'}_{n},\Upsilon'^{0}_{n}\big)$ be the decomposable–Henkin model formed from the family $\{ {\cal E}_{j}\colon j\in J\}$ by the ultrafilter ${\cal G'}\times{\cal H}_{1}\times\cdots\times{\cal H}_{n-1}\times{\cal H'}_{n}$, where ${{\cal B'}}_{n}$ is the $\tau$ reduct of $\Check{{\cal B'}}_{n}$ and put ${\cal B}_{n+1}={}^{J''_{n}}{\cal B'}_{n}/{\cal H''}_{n}$ so that ${\cal H}_{n}={\cal H'}_{n}\times {\cal H''}_{n}$. 
  \item[(4a)] For each $n\in\omega$ let
$
f_{n}: \Check{\cal A'}_{n}\longrightarrow \Check{\Check{\Check{{\cal B}}}}_{n+1}
$
be an isomorphism, where $\Check{\Check{\Check{{\cal B}}}}_{n+1}$ is a suitable extension of ${\Check{\Check{{\cal B}}}}_{n+1}$, with relations from $\Upsilon'_{n+1}$.
\item[(4b)] For  $0<n\in\omega$ let
$
g_{n}: \Check{\cal B'}_{n}\longrightarrow \Check{\Check{\Check{{\cal A}}}}_{n}$
be a suitable isomorphism, where $\Check{\Check{\Check{{\cal A}}}}_{n}$ is an extension of ${\Check{\Check{{\cal A}}}}_{n}$ with relations from $\Upsilon_{n}$.

\item[(5a)] For every $0<n<\omega$ $
\mathrm{Th}_{SO}\big({\Check{\Check{\Check{{\cal B}}}}_{n}}, \Upsilon^{0}_{n}\big)
=\mathrm{Th}_{SO}\big({{{\Check{{\cal A'}}}}_{n-1}}, \Upsilon'_{n-1}\big).
$
  \item[(5b)] For every $n\in\omega$ $
\mathrm{Th}_{SO}\big({\Check{\Check{\Check{{\cal A}}}}_{n}}, \Upsilon_{n}\big)
=\mathrm{Th}_{SO}\big({{{\Check{{\cal B'}}}}_{n}}, \Upsilon^{0'}_{n}\big).$

   \item[(6a)] For $0<n<\omega$ let   $\Check{{\cal A}}_{n}=\big\langle\Check{\Check{\Check{{\cal A}}}}_{n},R\big\rangle_{R\in \Upsilon_{n}},$ so that the newly added relations are denoted by relation symbols distinct from those introduced so far.
  \item[(6b)] For $0<n<\omega$ let   $\Check{{\cal B}}_{n}=\big\langle\Check{\Check{\Check{{\cal B}}}}_{n},S\big\rangle_{S\in \Upsilon^{0}_{n}},$ so that the newly added relations are denoted by relation symbols distinct from those introduced so far.

\end{itemize}
First, by our assumptions,
$
\mathrm{Th}_{SO}\big({\cal A}_{0}, \Upsilon_{0}\big)
=\mathrm{Th}_{SO}\big({{{{{\cal B}}}}_{0}}, \Upsilon^{0}_{0}\big)$
and let $\Check{{\cal A}}_{0}=\big\langle{\cal A}_{0},R\big\rangle_{R\in \Upsilon}$.  Then, by Claim~\ref{claim: t}, there exist a set $J_{0}$ and an ultrafilter ${\cal H}_{0}$ on $J_{0}$ such that for every $j\in J_{0}$ and  $R\in\Upsilon_{0}$ there is a relation $R^{j}\in\Upsilon_{0}^{0}$, and there exist ultraproducts
\[
\Check{\mathcal{A'}}_{0}={}^{I'_{0}}\Check{\mathcal{A}}_{0}/{\cal F'}_{0},\qquad
\Check{\Check{\Check{{\cal B}}}}_{1}
=\prod_{j\in J_{0}}\big\langle{\mathcal{B}}_{0},R^{j}\big\rangle_{R\in\Upsilon_{0}}/{\cal H}_{0},
\]
and an isomorphism
\[
f_{0}:\Check{\mathcal{A'}}_{0}\longrightarrow\Check{\Check{\Check{{\cal B}}}}_{1},
\]
such that
$$
\mathrm{Th}_{SO}\big(\Check{\mathcal{A'}}_{0},\Upsilon'_{0}\big)
=\mathrm{Th}_{SO}\big(\Check{\Check{\Check{{\cal B}}}}_{1},\Upsilon^{0}_{1}\big).$$
Thus, the above conditions are all satisfied for $f_{0}$ and for the newly defined structures.

From this point onward, the construction is symmetric with respect to the ultrachains $\langle \mathcal{A}_{n}:n\in\omega\rangle$ and $\langle \mathcal{B}_{n}:n\in\omega\rangle$.  Thus, for the general step, suppose that we have already defined $\Check{\cal B}_{n}$, ${\Check{\Check{\Check{{\cal A}}}}_{n}}$, $\Check{\cal B'}_{n}$, $g_{n}$, as well as $\Upsilon_{n}$, $\Upsilon^{0}_{n}$ and $\Upsilon'^{0}_{n}$, so that the conditions (1a)–(6a) and (1b)–(6b) hold for these and for the previously defined objects.  In particular assume
$
\mathrm{Th}_{SO}\big({\Check{\Check{\Check{{\cal A}}}}_{n}}, \Upsilon_{n}\big)
=\mathrm{Th}_{SO}\big({\Check{{\cal B'}}}_{n}, \Upsilon'^{0}_{n}\big).
$\[\begin{tikzcd}[ampersand replacement=\&,cramped]
	{{{{\Check{{\cal A}}}}'_{n}}} \& {{\Check{\Check{\Check{{\cal B}}}}_{n+1}}} \\
	{{\Check{\Check{\Check{{\cal A}}}}_{n}}} \& {{{{\Check{{\cal B}}}}'_{n}}} \\
	\& {{{{\Check{{\cal B}}}}_{n}}}
	\arrow["{f_{n}}", dashed, from=1-1, to=1-2]
	\arrow["{\delta'_{\Check{{\cal A}}_{n}}}", dashed, from=2-1, to=1-1]
	\arrow["{\delta_{\Check{{\cal B}}'_{n}}}"', dashed, from=2-2, to=1-2]
	\arrow["{g_{n}}", from=2-2, to=2-1]
	\arrow["{\delta'_{\Check{{\cal B}}_{n}}}"', from=3-2, to=2-2]
\end{tikzcd}\]
Then by Claim~\ref{claim: t}, there exist a set $J''_{n}$ and an ultrafilter ${\cal H''}_{n}$ on $J''_{n}$ such that for every $j\in J''_{n}$ and every $R\in\Upsilon_{n}$ there is a relation $R^{j}\in\Upsilon_{n}^{0'}$, and there exist ultraproducts
\[
\Check{\mathcal{A'}}_{n}={}^{I'_{n}}\Check{\mathcal{A}}_{n}/{\mathcal{F'}_{n}},\qquad \Check{\Check{\Check{{\cal B}}}}_{n+1}
=\prod_{j\in J''_{n}}\big\langle\Check{\mathcal{B'}}_{n},R^{j}\big\rangle_{R\in\Upsilon_{n}}/{\cal H''}_{n}
\]
and an isomorphism
\[
f_{n}:\Check{\mathcal{A'}}_{n}\longrightarrow\Check{\Check{\Check{{\cal B}}}}_{n+1},
\]
such that
\[
\mathrm{Th}_{SO}\big(\Check{\Check{\Check{{\cal B}}}}_{n+1}, \Upsilon^{0}_{n+1}\big)
=\mathrm{Th}_{SO}\big(\Check{\mathcal{A'}}_{n}, \Upsilon'_{n}\big),
\]
and the conditions (1a)–(6a) and (1b)–(6b) continue to hold.

We claim that the following two diagrams commute for every $0<n<\omega$ and every $m\in\omega$, i.e.
$
\delta_{\Check{{\cal B}}_{n}}
= f_{n}\circ \delta'_{\Check{\Check{\Check{{\cal A}}}}_{n}}\circ g_{n} \circ \delta'_{\Check{{\cal B}}_{n}},
$
and
$
\delta_{\Check{{\cal A}}_{m}}
= g_{m+1}\circ \delta'_{\Check{\Check{\Check{{\cal B}}}}_{m+1}}\circ f_{m} \circ \delta'_{\Check{{\cal A}}_{n}}.
$

\[\begin{tikzcd}[ampersand replacement=\&,cramped,column sep=scriptsize,row sep=small]
	\&\& {{\Check{\Check{\Check{{\cal B}}}}_{n+1}}} \&\&\& {{\Check{\Check{\Check{{\cal A}}}}_{m+1}}} \& {{\Check{{\cal B'}}_{m+1}}} \& {{\Check{\Check{\Check{{\cal B}}}}_{m+1}}} \\
	\& {{\Check{{\cal A'}}_{n}}} \&\&\&\&\& {{\Check{{\cal A'}}_{m}}} \\
	{{\Check{\Check{\Check{{\cal A}}}}_{n}}} \& {\Check{{\cal B'}}_{n}} \& {\Check{{\cal B}}_{n}} \&\&\& {{\Check{{\cal A}}_{m}}}
	\arrow["{g_{m+1}}"', from=1-7, to=1-6]
	\arrow["{\delta'_{{\Check{\Check{\Check{{\cal B}}}}_{m+1}}}}"', from=1-8, to=1-7]
	\arrow["{f_{n}}", from=2-2, to=1-3]
	\arrow["{f_{m}}"', from=2-7, to=1-8]
	\arrow["{\delta'_{{\Check{\Check{\Check{{\cal A}}}}_{n}}}}", from=3-1, to=2-2]
	\arrow["{g_{n}}", from=3-2, to=3-1]
	\arrow["{\delta_{{\Check{{\cal B}}_{n}}}}"', from=3-3, to=1-3]
	\arrow["{\delta'_{{\Check{{\cal B}}_{n}}}}", from=3-3, to=3-2]
	\arrow["{\delta_{{\Check{{\cal A}}_{m}}}}", from=3-6, to=1-6]
	\arrow["{\delta'_{{\Check{{\cal A}}_{m}}}}"', from=3-6, to=2-7]
\end{tikzcd}\]
Our claim is symmetric, so it suffices to prove
$
\delta_{\Check{{\cal B}}_{n}}
= f_{n}\circ \delta'_{\Check{\Check{\Check{{\cal A}}}}_{n}}\circ g_{n} \circ \delta'_{\Check{{\cal B}}_{n}}.$
To see this, note that every singleton relation $R_{b}=\{b\}$ (for $b\in B_{n}$) belongs to $\Upsilon^{0}_{n}$, hence $R_{b}$ appears in ${\Check{{\cal B}}_{n}}$. Therefore
\[
\delta_{\Check{{\cal B}}_{n}}(b)
= \delta^{*}_{\Check{{\cal B}}_{n}}(R_{b}) 
= f_{n}^{*}\big(\delta'^{*}_{\Check{\Check{\Check{{\cal A}}}}_{n}}\big(g_{n}^{*}\big(\delta'^{*}_{\Check{{\cal B}}_{n}}(R_{b})\big)\big)\big)
= f_{n}\big(\delta'_{\Check{\Check{\Check{{\cal A}}}}_{n}}\big(g_{n}\big(\delta'_{\Check{{\cal B}}_{n}}(b)\big)\big)\big),
\]
as required and hence we obtain that for every $n\in\omega$,
$f_{n}\circ \delta'_{{\cal A}_{n}}\subseteq f_{n+1}\circ \delta'_{{\cal A}_{n+1}} .$

Let
\[
f=\bigcup_{n\in\omega}f_{n}\circ \delta'_{{\cal A}_{n}}.
\]
Then it is easy to check that
\[
f:\mathcal{A}_{\omega}\longrightarrow\mathcal{B}_{\omega}
\]
is an isomorphism which induces a bijection between the decomposable relations of $\mathcal{A}_{\omega}$ and  $\mathcal{B}_{\omega}$ (where $\mathcal{A}_{\omega}$ and $\mathcal{B}_{\omega}$ are the limits of the ultrachains $\langle \mathcal{A}_{n}:n\in\omega\rangle$ and $\langle \mathcal{B}_{n}:n\in\omega\rangle$, respectively). This completes the proof of the Lemma.
\end{proof}
As a corollary of the previous theorem, we obtain the following characterization of structures that are equivalent in second-order logic.

\begin{theorem}\label{thm: eqv}
Let ${\cal A}$ and ${\cal B}$ be two $\tau$-structures. The following are equivalent:
\begin{itemize}
\item[(1)] ${\cal A}\equiv_{SO}{\cal B}$, i.e.\ the same second-order formulas hold in both structures;
\item[(2)]  $\{\cal A\}$ and $\{\cal B\}$ are completely inseparable from each other.
\end{itemize}
\end{theorem}

\begin{proof}
This follows from Lemma~\ref{lem: főlemma} and the elementary fact that the only ultrafilter on a singleton set is trivial.
\end{proof}

The previous theorem may be regarded as an analogue of Frayne's theorem, where he showed that two structures are elementarily equivalent (in first-order logic) iff they have isomorphic $\omega$-long ultrachains \cite{Frayne}. As mentioned above, Keisler and Shelah later substantially simplified this characterization: two structures are elementarily equivalent iff they have isomorphic ultrapowers \cite{Shelah}. It is natural to ask whether the characterization given in Theorem~\ref{thm: eqv} admits a similarly simple form.

\begin{problem}
Is it true that ${\cal A}\equiv_{SO}{\cal B}$  if and only if there exist ultrapowers ${\cal A}'={}^{I}{\cal A}/{\cal F}$ and ${\cal B}'={}^{I}{\cal B}/{\cal F}$ and an isomorphism
\[
f:{\cal A}'\longrightarrow {\cal B}'
\]
such that $f$ induces a bijection between the decomposable relations of ${\cal A}'$ and those of ${\cal B}'$?
\end{problem}

Now we turn to the characterization of classes axiomatizable in second-order logic.  First, consider the following definition.

\begin{definition}\label{def: zárt-teljesinszep}Let $K$ be a class of $\tau$-structures.
\begin{itemize}
\item[(i)] We say that $K$ is \emph{closed under semi-complete inseparability} if for every set $\{ {\cal A}_{i} : i\in I\}\subseteq K$ and  $\tau$-structure ${\cal B}$, whenever $\{ {\cal A}_{i} : i\in I\}$ and $\{ {\cal B}\}$ are completely inseparable from each other, then ${\cal B}\in K$.
\item[(ii)] We say that $K$ is \emph{closed under complete inseparability} if for every sets $\{ {\cal A}_{i} : i\in I\}\subseteq K$ and  $\{ {\cal B}_{j} : j\in J\}$ of $\tau$-structures, whenever $\{ {\cal A}_{i} : i\in I\}$ and $\{ {\cal B}_{j} : j\in J\}$ are completely inseparable from each other, then $$\{ {\cal B}_{j} : j\in J\}\cap K\neq\emptyset.$$
\end{itemize}

\end{definition}
Note that the previous definitions are purely algebraic.  As a conclusion of this section we give a purely algebraic characterization of those classes that are definable by second-order formulas, and of those classes that are finitely definable by second-order formulas.
\begin{theorem}\label{thm: axi so}
    Let $K$ be a class of $\tau$-structures. Then $K$ is definable by second-order formulas (that is, there exists $\Gamma\subseteq SO$ such that $\mathrm{Mod}(\Gamma)=K$) if and only if $K$ is closed under semi-complete inseparability.

\end{theorem}

\begin{proof}
    First suppose that there exists $\Gamma\subseteq SO$ such that $\mathrm{Mod}(\Gamma)=K$.  Let $\{ {\cal A}_{i} : i\in I\}\subseteq K$ and  ${\cal B}$ be any $\tau$-structure such that $\{ {\cal A}_{i} : i\in I\}$ and $\{ {\cal B}\}$ are completely inseparable from each other. Then, by Lemma~\ref{lem: főlemma}, there is an ultrafilter $\cal G$ over $I$ such that
$$
\mathrm{Th}_{SO}({\cal B})=\{\phi\in SO\text{ : }\{i\in I\text{ : } {\cal A}_{i}\models\phi\}\in {\cal G}\},
$$
and since $\Gamma \subseteq \{\phi\in SO\text{ : }\{i\in I \text{ : } {\cal A}_{i}\models\phi\}\in {\cal G}\}$ we obtain ${\cal B}\in \mathrm{Mod}(\Gamma)$, hence ${\cal B}\in K$. Thus $K$ is closed under semi-complete inseparability.

Now suppose that $K$ is closed under semi-complete inseparability. Let
$$
\Lambda=\{\phi\in SO\text{ : }K\models \phi\}.
$$
We claim that $\Lambda$ axiomatizes $K$. Indeed, let ${\cal B}\in \mathrm{Mod}(\Lambda)$. Let $\Sigma= [\mathrm{Th}_{SO}({\cal B})]^{<\omega},$ then for every $i\in\Sigma$ there exists ${\cal A}_{i}\in K$ such that $${\cal A}_{i}\models \bigwedge i$$ (otherwise $K\models \neg i$, so $\neg i\in \Lambda$, contradicting ${\cal B}\in \mathrm{Mod}(\Lambda)$). We claim that the family $\{ {\cal A}_{i}: i\in \Sigma\}$ and $\{ {\cal B}\}$ are completely inseparable. To see this, let $\cal G$ be an ultrafilter on $\Sigma$ containing the sets $i^{\partial}=\{j\in\Sigma \text{ : }\; i\subseteq j\}$ for $i\in\Sigma$. Then
$$
\mathrm{Th}_{SO}({\cal B})=\{\phi\in SO\text{ : }\{i\in \Sigma\text{ : } {\cal A}_{i}\models\phi\}\in {\cal G}\},
$$
and hence by Lemma~\ref{lem: főlemma} the family $\{ {\cal A}_{i}: i\in \Sigma \}$ and $\{ {\cal B}\}$ are completely inseparable. By our hypothesis, it follows that ${\cal B}\in K$. Therefore $K=\mathrm{Mod}(\Lambda)$, which completes the proof.
\end{proof}

\begin{theorem}\label{thm: axi finit so}
    Let $K$ be a class of $\tau$-structures. Then $K$ is definable by a second-order formula (that is, there exists $\phi\in SO$ such that $\mathrm{Mod}(\phi)=K$) if and only if $K$ is closed under complete inseparability.

\end{theorem}

\begin{proof}
    First suppose that there exists $\phi\in SO$ such that $\mathrm{Mod}(\phi)=K$, and let $\{ {\cal A}_{i} : i\in I\}\subseteq K$ and $\{ {\cal B}_{j} : j\in J\}$ be sets of $\tau$-structures such that $\{ {\cal A}_{i} : i\in I\}$ and $\{ {\cal B}_{j} : j\in J\}$ are completely inseparable from each other. Then, by Lemma~\ref{lem: főlemma}, there are ultrafilters ${\cal G}$ over $I$ and ${\cal G'}$ over $J$ such that
\[
\{\xi\in SO\text{ : }\{i\in I\text{ : }{\cal A}_{i}\models \phi\}\in {\cal G}\}
=
\{\psi\in SO\text{ : }\{j\in J\text{ : }{\cal B}_{i}\models \psi\}\in {\cal G'}\}.
\]
Hence
\[
\emptyset\neq\{j\in J\text{ : }{\cal B}_{i}\models \phi\}\in {\cal G'},
\]
so
\[
\{ {\cal B}_{j} : j\in J\}\cap \mathrm{Mod}(\phi) \neq\emptyset,
\]
and therefore $K$ is closed under complete inseparability.

Now suppose, towards a contradiction, that $K$ is closed under complete inseparability but there is no $\phi\in SO$ with $\mathrm{Mod}(\phi)=K$. Then, by Lemma~\ref{lem: 4eqv}, $T_{\overline{K}}\cap T_{\overline{K^{c}}}\neq\emptyset$, i.e.\ there exist a decomposable–Henkin model $\big({\cal A}, \Upsilon\big)$ formed from the family $\{ {\cal A}_{i}\text{ : } i\in I\}\subseteq K$ by an ultrafilter $\cal G$ and a decomposable–Henkin model $\big({\cal B}, \Upsilon^{0}\big)$ formed from the family $\{ {\cal B}_{j}\text{ : } j\in J\}\subseteq K^{c}$ by an ultrafilter $\cal G'$, such that
\[
\mathrm{Th}_{SO}\big({\cal A}, \Upsilon\big)= \mathrm{Th}_{SO}\big({\cal B}, \Upsilon^{0}\big).
\]
Hence,  by Lemma~\ref{lem: Łos},
\[
\{\xi\in SO\text{ : }\{i\in I\text{ : }{\cal A}_{i}\models \xi\}\in {\cal G}\}
=
\{\psi\in SO\text{ : }\{j\in J\text{ : }{\cal B}_{j}\models \psi\}\in {\cal G'}\},
\]
and therefore, by Lemma~\ref{lem: főlemma}, the families $\{ {\cal A}_{i}\text{ : } i\in I\}$ and $\{ {\cal B}_{j}\text{ : } j\in J\}$ are completely inseparable from each other. By the hypothesis that $K$ is closed under complete inseparability, it follows that $\{ {\cal B}_{j}\text{ : } j\in J\}\cap K\neq\emptyset$, which is impossible since $\{ {\cal B}_{j}\text{ : }j\in J\}\subseteq K^{c}$. This contradiction proves the theorem.\end{proof}

\begin{center}
    { \textbf{Acknowledgements}}
\end{center} 
I am deeply grateful to Gábor Sági for his  constant care  and support throughout this work.  I thank him for everything I have learned from him and for suggesting this fascinating topic.  He is a fantastic and inspiring teacher, and I wish everyone could have a supervisor as excellent as he has been for me.

\bigbreak \leftline{Eotvös Loránd University, }
\leftline{Institute of Mathematics,} \leftline{Pázmány Péter stny.
1/C, 1117 Budapest, Hungary}
\leftline{e-mail: janos.ivanyos@gmail.com} \ \\
\\

\end{document}